\documentclass[11pt,a4paper,reqno]{amsart}

\usepackage{times}
\usepackage{amsmath}
\usepackage{amssymb}
\usepackage{amsthm}
\usepackage{latexsym}
\usepackage{amsfonts}
\usepackage{xcolor}
\usepackage{graphicx}
\DeclareGraphicsExtensions{.eps,.pdf,.png,.jpg}
\usepackage{bbm}
\usepackage{tikz-cd}
\usepackage{adjustbox}
\usepackage[colorlinks=true, citecolor=blue, linkcolor=blue, anchorcolor=black]{hyperref}
\usepackage{cases}
\usepackage[normalem]{ulem}

\newtheorem{thm}{Theorem}[section]

\newtheorem{lem}[thm]{Lemma}
\newtheorem{prop}[thm]{Proposition}
\theoremstyle{definition}

\newtheorem{rem}[thm]{Remark}

\newtheorem{ass}[thm]{Assumption}

\numberwithin{equation}{section}
\allowdisplaybreaks

\newcommand{\innerprod}[2]{\left\langle #1,\, #2 \right\rangle} 

\newcommand{\rmref}[1]{{\rm\ref{#1}}}

\newcommand{\linspan}{\operatorname{span}}
\newcommand{\ran}{\operatorname{ran}}
\newcommand{\dom}{\operatorname{dom}}
\newcommand{\ol}{\overline}

\newcommand{\wt}{\widetilde}
\newcommand{\wh}{\widehat}
\newcommand{\<}{\langle}
\renewcommand{\>}{\rangle}

\newcommand{\R}{\ensuremath{\mathbb R}}    
\newcommand{\C}{\ensuremath{\mathbb C}}    
\newcommand{\N}{\ensuremath{\mathbb N}}    

\newcommand{\calB}{\mathcal B}         
         
\newcommand{\calD}{\mathcal D}         
\newcommand{\calE}{\mathcal E}

\newcommand{\calH}{\mathcal H}

\newcommand{\calK}{\mathcal K}         
\newcommand{\calL}{\mathcal L}

\newcommand{\calO}{\mathcal O}

\newcommand{\calU}{\mathcal U}

\newcommand{\calX}{\mathcal X}         
\newcommand{\calY}{\mathcal Y}         
\newcommand{\calZ}{\mathcal Z}         

				 \newcommand{\bE}{\mathbb E}

				 \newcommand{\bH}{\mathbb H}

				 \newcommand{\bK}{\mathbb K}

				 \newcommand{\bP}{\mathbb P}

				 \newcommand{\bV}{\mathbb V}

\newcommand{\la}{\lambda}
\newcommand{\veps}{\varepsilon}

\newcommand{\vphi}{\varphi}

\newcommand{\Tr}{\operatorname{Tr}}

\newcommand{\diag}{\operatorname{diag}}

\newcommand{\FN}[1]{{\color{red}#1}}
\newcommand{\nc}{{n_u}}

\title[Error analysis of kernel EDMD for prediction and control in the Koopman framework]{Error analysis of kernel EDMD for prediction and control\\ in the Koopman framework}

\begin{document}
\begin{abstract}
Extended Dynamic Mode Decomposition (EDMD) is a popular data-driven method to approximate the Koopman operator for deterministic and stochastic (control) systems. This operator is linear and encompasses full information on the (expected stochastic) dynamics. In this paper, we analyze a kernel-based EDMD algorithm, known as kEDMD, where the dictionary consists of the canonical kernel features at the data points. The latter are acquired by i.i.d.\ samples from a user-defined and application-driven distribution on a compact set. 
We prove bounds on the prediction error of the kEDMD estimator when sampling from this (not necessarily ergodic) distribution. 
The error analysis is further extended to control-affine systems, where the considered invariance of the Reproducing Kernel Hilbert Space is significantly less restrictive in comparison to invariance assumptions on an a-priori chosen dictionary.
\end{abstract}


\author[F.~Philipp]{Friedrich Philipp}
\address{{\bf F.~Philipp}
	Technische Universit\"at Ilmenau, Optimization-based Control Group, Institute for Mathematics, Ilmenau, Germany}
\email{friedrich.philipp@tu-ilmenau.de}

\author[M.~Schaller]{Manuel Schaller}
\address{{\bf M.~Schaller}
	Technische Universit\"at Ilmenau, Optimization-based Control Group, Institute for Mathematics, Ilmenau, Germany}
\email{manuel.schaller@tu-ilmenau.de}

\author[K.~Worthmann]{Karl Worthmann}
\address{{\bf K.~Worthmann}
	Optimization-based Control Group, Institute of Mathematics, Technische Universit\"at Ilmenau, Ilmenau, Germany}
\email{karl.worthmann@tu-ilmenau.de}

\author[S.~Peitz]{Sebastian Peitz}
\address{{\bf S.~Peitz}
	Paderborn University, Department of Computer Science, Data Science for Engineering, Germany}
\email{sebastian.peitz@upb.de}

\author[F.~N\"uske]{Feliks N\"uske}
\address{{\bf F.~N\"uske}
	Max Planck Institute for Dynamics of Complex Technical Systems, Magdeburg, Germany\\
 and Department of Mathematics and Computer Science, Freie Universität Berlin, Germany}
\email{nueske@mpi-magdeburg.mpg.de}

\maketitle

\section{Introduction}

\noindent Extended Dynamic Mode Decomposition (EDMD, see, e.g., \cite{williams15}) and its predecessor DMD~\cite{Schm10} have received much attention in recent years. 
The reason is the capability to predict so-called \textit{observables} along the flow of dynamical (control) systems in a data-driven manner. 
These observable functions are quantities of interests, which can be used for prediction~\cite{klus16}, computation and analysis of key characteristics like attractors~\cite{Kueh21} or metastable sets~\cite{nueske14,klus18_rev,nueske21} as well as control~\cite{kaiser21}, cf.\ the book~\cite{mauroy20} and the recent overview article~\cite{brunton22} for further 
applications. 
The particular appeal of the EDMD approach 
is due to several reasons: First, the implementation of the learning scheme is comparatively simple. 
Second, the resulting model is linear and, last, the Koopman framework provides a mathematically-sound foundation. 

Application of the EDMD method usually requires two steps: First, the linear Koopman operator \cite{koo31} (or its generator) is restricted to a subspace $\calD$ spanned by a finite set of observable functions. Second, the projection of this finite-rank linear operator to $\calD$, also called \textit{compression}, is approximated using samples, i.e., evaluations of the observables 
along the flow of the dynamical system. The first step usually requires the a-priori construction of a dictionary consisting of 
functions on the state space such as, e.g., monomials, which inevitably leads to very large dictionaries for high-dimensional state spaces. 
Consequently, the respective data requirements to learn the 
compression of the Koopman operator grows rapidly due to the curse of dimensionality. 
Data-informed dictionaries using kernel methods have proven their potential to alleviate this issue. 
In kernel EDMD (kEDMD), the dictionary is built up by kernel feature evaluations on the data set~\cite{williams15_kernel,klus20} which are contained in an infinite-dimensional ambient Reproducing Kernel Hilbert Space (RKHS, see, e.g., \cite{pr,sc}). Thus, the a-priori choice and tuning of the dictionary reduces to the mere choice of a suitable kernel. 
As another advantage, one gains access to powerful tools from the rich theory of RKHSs such as the 
kernel trick. For further recent works on the use of reproducing kernels in the context of dynamical systems see also \cite{bblshh,das2021,kostic22,kostic23}.

To derive data-driven approximations of Koopman operator or 
generator, two fundamentally different sampling techniques are common. The first one is to draw independent-and-identically-distributed (i.i.d.) initial values and to propagate the flow by one time step. In this way, the dynamics 
is learned by means of more and more i.i.d~samples 
and results from, e.g., Monte-Carlo integration, can be invoked. 
The second approach, which is very appealing from a practical point of view, is to aqcuire data from one (or several) long trajectories. 
This approach is called ergodic sampling and is widely applied to systems stemming, for instance, from statistical physics problems, see~\cite{lelievre10_fe,lelievre16}. 
As these measurements are clearly not independent, the exploration of the state space has to be ensured by means of properties of the dynamical system such as ergodicity. 
Ergodic sampling is often mandated if the target measure is a high-dimensional multi-modal distribution, where efficient algorithms to generate i.i.d samples are unavailable. It is also appealing due to its ease of implementation and its ability to reveal unknown ergodic measures and long-term behavior of the system. However, i.i.d.~sampling also has a particular advantage: 
In contrast to ergodic sampling, where one usually obtains characteristics of the behavior through the lens of the ergodic invariant measure, e.g., on an attractor of the system, i.i.d.~sampling allows for learning the dynamics on any desired region of the state space. 

For an a-priori chosen dictionary, i.e., the original EDMD approach, the approximation error was analyzed in~\cite{nueske23,zuazua21} for systems governed by ordinary and stochastic differential equations. 
However, for kEDMD, 
this analysis is so far only available if some invariant measure exists on the whole state space, see~\cite{PhilScha23}. 
In this work, we complement these works by extending the results of~\cite{PhilScha23} on the kernel- and data-based approximation of the Koopman operator to a much more flexible setting without (potentially restrictive) invariance assumptions on the sets and/or measures under consideration. 
Our main tool is a probabilistic error bound for (cross-)covariance operators in the Hilbert-Schmidt norm using Hoeffding's inequality. 
The non-invariance of the considered subset of the state space leads to Koopman operators which no longer map the function space into itself, but onto another space in general. 
Thus, the Koopman operators do not form a semigroup as it is the case under invariance assumptions. 
Hence, we generalize the definition of the Koopman generator to establish our findings without modifying the kernel-based EDMD algorithm.

To model control systems in the Koopman framework, a popular approach consists of seeking a linear surrogate model in a lifted state space by means of augmenting the state by the control, cf.\ \cite{KordMezi18}. 
Bi-linear models are an alternative~\cite{WillHema16,Sura16}. This is motivated by the fact that 
the Koopman generator inherits control affinity and can, thus, be constructed using Koopman generators of autonomous systems. While loosing linearity of the data-driven surrogate model, the bi-linear approach has proven to be favorable for nonlinear systems with state-control couplings~\cite{PeitOtto2020}. 
For EDMD, a respective analysis of the approximation error for dynamical control systems can be conducted by splitting the error into its two sources, cf.\ \cite{nueske23} for the estimation error and \cite{SchaWort22} for the projection error. 
In these works, error bounds for the generator approximation of autonomous systems were leveraged exploiting the control-affine structure. 
In addition, for set-point stabilization, an error bound proportional to the size of the state and the control was established in~\cite{BoldGrun23}.
On the one hand, these error bounds were leveraged in~\cite{StraScha23} to design a robust controller of the bi-linear system using techniques based on linear matrix inequalities. On the other hand, they were exploited in~\cite{BoldGrun23} to rigorously ensure practical asymptotic stability of EDMD-based Model Predictive Control using cost controllability~\cite{CoroGrun20} of the original system. In this paper, we provide the first error bounds for the approximation of control-affine systems using kernel-based EDMD. Whereas many existing results concerning EDMD for control require the Koopman invariance of the finite-dimensional linear span of the pre-defined dictionary~\cite{GoswPale21,BoldGrun23}, we are able to improve our approximation bounds by assuming that the Koopman operator leaves the (usually infinite-dimensional) ambient RKHS invariant. In contrast to the standard EDMD case with a finite-dimensional dictionary space, this assumption appears to be much more naturally satisfied in many relevant cases.


\section{Summary of the key results}\label{s:summary}
We briefly highlight our two main findings. 
Our first result considers error bounds for kernel EDMD (kEDMD) of 
systems governed by a stochastic differential equation
\begin{equation*}
    dX_t = b(X_t)\,dt + \sigma(X_t)\,dW_t
\end{equation*}
for initial values contained in a compact set $\calX$. The Koopman operator is defined by the conditional expectation of observables along the flow, which is assumed to be contained in possibly unbounded set~$\calY$. 
More precisely, for an observable function $\psi : \calY \to \R$ and for all $x\in \calX$, we define
\begin{align}\label{eq:koop}
    (K^t\psi)(x) = \bE\big[\psi(X_t)|X_0=x\big]. 
\end{align}
We explicitly allow for $\sigma\equiv 0$ such that ordinary differential equations are included. To approximate the Koopman operator from data, we collect samples $(x_k,y_k)_{k=0}^{m-1}$, where $x_k \in \calX$ is distributed according to a chosen 
Borel measure~$\mu$ on~$\calY$ such that $\mu (\calX) = 1$ holds and $y_k \sim \rho_t(x_k,\cdot)$ denotes the corresponding successor state. Here, $\rho_t(x,\cdot)$ is the probability distribution associated with the transition kernel of the process, i.e., $\rho_t(x,A) = \bP(X_t\in A\,|\,X_0=x)$. 
As we are interested in approximations of the Koopman operator on Reproducing Kernel Hilbert Spaces (RKHS) on $\calX$ and $\calY$ denoted by $\bH_\calX$ and $\bH_\calY$, respectively, the data matrices $\Psi(X)$ and $\Psi(Y)$ are then formed by kernel evaluations at the data points, i.e.,
\[
\Psi(X) = \big(k(x_i,x_j)\big)_{i,j=0}^{m-1}
\qquad\text{and}\qquad
\Psi(Y) = \big(k(x_i,y_j)\big)_{i,j=0}^{m-1}.
\]
Correspondingly, the kEDMD estimator $\wh K_r^{m,t}$ of the Koopman operator $K^t$ on the function space $\calD$ spanned by the kernel features, i.e., $\calD = \linspan\{k(x_0,\cdot),\ldots,k(x_{m-1},\cdot)\}$, is given by 
\begin{align*}
\wh K_r^{m,t}:\calD \to \calD  \quad \text{with\ matrix\ representation}\quad  \wh M^{m,t}_r = [\Psi(X)]^\dagger_r \Psi(Y)^\top,    
\end{align*}
where $[\Psi(X)]^\dagger_r$ is a pseudo-inverse of a rank-$r$ truncation of $\Psi(X)$ to avoid ill-conditioning.

\smallskip
\noindent \textbf{First main result: error bound for autonomous systems}, cf.\ Theorem~\ref{t:main}. Let an error bound $\varepsilon > 0$ and a probabilistic tolerance $\delta \in (0,1)$ be given. Then, under suitable assumptions to be specified later, there is a sufficient amount of data $m_0\in \N$ such that for all $m \geq m_0$,
\begin{equation}\label{eq:intest1}
\|P_rK^t - \wh K_r^{m,t}\|_{\bH_\calY\to L^2(\calX;\mu)}\,\le\,c_r\veps
\end{equation}
holds with probability $1-\delta$. Here, $c_r\geq 0$ is a constant depending on the truncation order $r\in \N$, and $P_r$ is an orthogonal projection onto the first $r$ components of the Mercer orthogonal basis of $L^2(\calX;\mu)$.

Further, if the chosen RKHS is invariant under the Koopman flow, then, with probability at least $1-\delta$,
\begin{equation*}
\|K^t - \wh K_r^{m,t}\|_{\bH_\calY\to L^2(\calX;\mu)}\,\le\,\sqrt{\la_{r+1}}\,\|K^t\|_{\bH_\calY\to\bH_\calX} + c_r\veps
\end{equation*}
with Mercer eigenvalues $\lambda_r\to 0$ for $r\to \infty$.

As a second result, we provide bounds for kEDMD-based predictions of control-affine systems
\begin{align}\label{eq:contsys}
\dot{x}(t) = f(x(t)) + \sum_{i=1}^{\nc} u_i(t) g_i(x(t))
\end{align}
with a piecewise constant control function taking values in a compact set $\calU\subset \R^\nc$, resulting e.g.~from sampling with zero-order hold. For a constant control function $u(t) \equiv u \in \calU$, 
the counterpart of the Koopman operator~\eqref{eq:koop} 
is given by
\begin{align}\label{eq:koopcont}
    (K_u^t \psi)(x^0) := \psi\big(x(t;x^0,u)\big),
\end{align}
where $x^0\in\calX$ 
and $x(\cdot;x^0,u)$ is the solution to \eqref{eq:contsys} with $x(0) = 0$. Again, $\calY$ is chosen in a way such that $x(t;x^0,u) \in \calY$ for $x^0\in\calX$. 
As the generator of the Koopman operator~\eqref{eq:koopcont} inherits the control affine structure of the underlying control system~\eqref{eq:contsys}, we show that the Koopman operator is approximately control affine (for small times $t\geq 0$). Then, we define its approximation by
\begin{align}\label{eq:contaffinecon}
\wh K^{m,t}_{r,u} := \wh K^{m,t}_{r} + \sum_{i=1}^{\nc} \frac{u_i}{\gamma_i}\big(\wh{K}^{m,t}_{r,\gamma_i e_i} - \wh{K}^{m,t}_{r}\big),
\end{align}
where $\wh K_{r,\gamma_ie_i}^{m,t}$, $\wh{K}^{m,t}_{r}$ are estimators of the Koopman operators corresponding to~\eqref{eq:contsys} for the constant controls $u=\gamma_ie_i$, $i=1,\ldots,\nc$, and $u=0$, respectively, using the kEDMD method as outlined above for autonomous systems.

\smallskip
\noindent \textbf{Second main result: error bound for control systems}, cf.\ Theorem~\ref{t:control}.
Let an error bound $\varepsilon > 0$, a probabilistic tolerance $\delta \in (0,1)$, and a control $u\in \calU\subset \R^m$, $\calU$ compact be given. Then, under suitable assumptions to be specified later, there is a sufficient amount of data $m_0\in \N$ such that for all $m \geq m_0$,
\[
\|P_rK_u^t - \wh K_{r,u}^{m,t}\|_{\bH_\calY\to L^2(\calX;\mu)}\,\le\,c_r\veps + t^2\cdot C_k\wh\la_r^{-1}.
\]
The first term resembles the right-hand side of~\eqref{eq:intest1} and the second term stems from the \textit{approximate control-affinity} of the Koopman operator. 

The outline of this paper is as follows. 
In the following 
section, the reader is introduced to the setting and fundamentals of this work, such as stochastic differential equations, Reproducing Kernel Hilbert spaces, Koopman operators, and kEDMD including the main assumptions. 
In Section \ref{sec:res1} we present and prove our main result, Theorem \ref{t:main}, on the approximation error on the kEDMD estimator for autonomous systems. In Section \ref{s:control} we first recall the affine dependence of the Koopman generator on the control for control-affine systems and show that the same holds approximately (for small sampling times) for the Koopman operator, resulting in a natural estimator via an affine combination of autonomous operators. Our second main result, Theorem \ref{t:control}, provides an error bound on this estimator. Finally, Section \ref{s:num} concludes the paper with an exhibition of numerical results.

\medskip
\section{Setting}\label{sec:setting}
In the following, we introduce the reader to the setting considered in this paper which involves stochastic differential equations (Section~\ref{subsec:sde}), Koopman operators (Section~\ref{subsec:koopman}), symmetric positive definite kernels and their associated reproducing kernel Hilbert spaces (Section~\ref{subsec:rkhs}), and the kEDMD approximation method (Section~\ref{subsec:kEDMD}). In Section \ref{subsec:ass} we provide our assumptions and discuss them in detail.


\subsection{Stochastic differential equations}\label{subsec:sde}
Let a stochastic differential equation (SDE) with drift vector field $b : \R^d\to\R^d$ and diffusion matrix field $\sigma : \R^d\to\R^{d\times d}$ be given, i.e.,
\begin{equation}\label{e:SDE0}
dX_t = b(X_t)\,dt + \sigma(X_t)\,dW_t,
\end{equation}
where $W_t$ is $d$-dimensional Brownian motion. We assume that both $b$ and $\sigma$ are Lipschitz-continuous. Then \cite[Theorem 5.2.1]{oe} guarantees the existence of a unique solution $(X_t)_{t\ge 0}$ to \eqref{e:SDE0}. The case $\sigma\equiv 0$ will be referred to as ``the deterministic case'' as the SDE then reduces to the ODE $\dot x = b(x)$.

The solution $(X_t)_{t\ge 0}$ of the SDE \eqref{e:SDE0} is a continuous time-homogeneous Markov process, whose transition kernel is denoted by $\rho_t : \R^d\times \calB\to\R$, where $\calB$ is the Borel $\sigma$-algebra on $\R^d$. Then $\rho_t(x,\cdot)$ is a probability measure for all $x\in\R^d$, and for each $A\in\calB$ we have that $\rho_t(x,A)$ is a representative of the conditional probability for $A$ containing $X_t$, given $X_0=x$, i.e.,
$$
\rho_t(x,A) = \bP(X_t\in A|X_0=x)\quad\text{for $\bP^{X_0}$-a.e. $x\in \R^d$},
$$
where $\bP^{X_0}$ denotes the law of $X_0$. Note that the law of $X_t$ is determined by
\begin{align}\label{e:inv_distr}
\bP^{X_t}(A) = \int\rho_t(x,A)\,d\bP^{X_0}(x),\qquad A\in\calB.
\end{align}
It is well known \cite[Section 1.3]{bakry} that the transition kernel satisfies the so-called {\em Chapman-Kolmogorov equation}
\begin{equation}\label{e:cke}
\rho_{t+s}(x,A) = \int\rho_s(y,A)\,\rho_t(x,dy),\qquad A\in\calB,\;x\in\R^d\;s,t\ge 0.
\end{equation}
As it is usual, we agree to write $\rho_t(x,dy)$ for $d\rho_t(x,\cdot)(y)$.

\subsection{Assumptions}\label{subsec:ass}
In many works it is assumed that there is some set $\calX \subseteq \R^d$ which is invariant under the flow of the SDE and that there is an invariant distribution $\pi$ on $\calX$ such that the solution process $X_t$ is stationary with respect to $\pi$, i.e., $\bP^{X_t}=\pi$ for all $t\ge 0$. In view of \eqref{e:inv_distr}, this is equivalent to
\[
\int\rho_t(x,A)\,d\pi(x) = \pi(A),\qquad A\in\calB_\calX,\,t\ge 0.
\]
Here, $\calB_\calX$ denotes the Borel sigma algebra on $\calX$.

However, although the Krylov–Bogolyubov Theorem \cite{dpz} ensures the existence of such an invariant measure $\pi$ for a large class of SDEs, both the set $\calX$ as the measure $\pi$ are often unknown in practice; even if $\pi$ is known, generating independent and identically distributed (i.i.d) samples with respect to it remains, in general, a challenging problem. 

In this paper, we relax both conditions in a way that the objects become both more flexible and more accessible for applications:
\begin{itemize}
\item $\calX := \ol{\Omega}$, where $\Omega\subset\R^d$ is a bounded and connected open set;
\item $T>0$ is a finite time horizon;
\item $\calY\subset\R^d$ is a closed set such that $\calX\subset\calY$ and
\begin{align}\label{e:stay_in_Y}
\forall t\ge 0\,\forall x\in\calX : \rho_t(x,\calY) = 1.
\end{align}
\item $\mu$ is a finite Borel measure on $\calY$ such that $\mu(\calX) = 1$ and there exists a constant $L>0$ such that
\begin{align}\label{e:LipMeas}
\int_\calX\rho_t(x,A)\,d\mu(x)\,\le\, L\cdot\mu(A),\quad A\in\calB_\calY,\,t\in [0,T].
\end{align}
\end{itemize}
The set $\calX$ will later serve as the set, which initial data for the process is sampled from. The assumption~\eqref{e:stay_in_Y} then assures that a subsequent sample after time $t\in (0,T]$ of the process remains in $\calY$ almost surely. Clearly, $\calY = \R^d$ is a possible choice.

Let us discuss the above assumptions in the following remark.

\begin{rem}\label{r:on_assumptions}
{\bf (a)} In the case $\calY = \calX$, the set $\calX$ is invariant under the stochastic flow of the SDE \eqref{e:SDE0}, i.e.,
$$
\forall t\ge 0\,\forall x\in\calX : \rho_t(x,\calX) = 1.
$$
A sufficient condition on the coefficients of the SDE \eqref{e:SDE0} for the invariance of $\calX$ under the SDE flow is given by 
\begin{subequations}
\begin{align}
\label{e:deter_outside}
\sigma(x)&=0\quad\text{for all $x\in\R^d\backslash\calX$ and}\\
\label{e:IP_condition}
b(x)^\top\nu(x) &\leq 0 \quad \text{for all $x\in\partial\calX$ and all outer normal vectors $\nu(x)\in \R^n$},
\end{align}
\end{subequations}
where an outer normal vector of $\calX$ at $x\in\partial\calX$ is any vector $\nu\in\R^d$, $\nu\neq 0$, such that the open ball with center $x + \nu$ and radius $\|\nu\|_2$ and $\calX$ are disjoint\footnote{If $\Omega$ has a $C^1$-boundary, the outer unit vector is unique. In case of Lipschitz continuous boundary, as a consequence of Rademacher's theorem \cite[Theorem 3.2]{EvansGariepy}, it is unique up to a set of Lebesgue measure zero on the boundary.}. This follows by an application of the Bony-Brezis (also known as Nagumo) theorem, which is proven in \cite[Theorem~10.XVI, p.\ 117f]{Walt1998} for the deterministic case under the \textit{tangent or interior pointing}-condition~\eqref{e:IP_condition}. In the stochastic setting, continuity of solutions to the SDE ensures that any trajectory leaving $\calX$ has to cross the boundary. The condition~\eqref{e:deter_outside} implies that the noise has no influence at the boundary ($\sigma$ is continuous) and in the exterior of $\calX$ such that the SDE reduces to an ordinary differential equation. Hence, the proof in  \cite{Walt1998} for the deterministic setting also applies to the SDE case.

\smallskip\noindent
{\bf (b)} Note that \eqref{e:LipMeas} means that the push-forward measures $\nu_t$ of $\mu$ under the dynamics, defined on $\calY$ by $\nu_t(A) = \int_\calX\rho_t(x,A)\,d\mu(x)$, $t\in [0,T]$, are uniformly Lipschitz-continuous w.r.t.\ $\mu$, i.e., $\nu_t(A)\le L\cdot\mu(A)$ for all $A\in\calB_\calY$, cf.\ \cite[Definition 3.1]{phil17}. In consequence, $\mu$-integrable functions on $\calY$ are also $\nu_t$-integrable, and for $\psi\in L^1(\calY;\mu)$ we have
\begin{equation}\label{e:LipMeas_Func}
\int_\calX\int_\calY|\psi(y)|\,\rho_t(x,dy)\,d\mu(x) = \int_\calY|\psi|\,d\nu_t\,\le\,L\int_\calY|\psi|\,d\mu.
\end{equation}

\smallskip\noindent
{\bf (c)} In the deterministic case (i.e., $\sigma\equiv 0$), \eqref{e:stay_in_Y} means that $F_t(x)\in\calY$ for all $t\in [0,T]$ and all $x\in\calX$, where $F_t$ denotes the flow of the ODE $\dot x = b(x)$. Moreover, the condition \eqref{e:LipMeas} is equivalent to
\begin{align}\label{e:LipMeas_det}
\mu(F_t^{-1}(A))\le L\cdot\mu(A),\qquad A\in\calB_\calY,\,t\in [0,T].
\end{align}
In Section \ref{a:deterministic} we show that this condition is satisfied for a large class of systems and measures.
\end{rem}

\subsection{Koopman operators}\label{subsec:koopman}
For $p\in [1,\infty]$, we let $\|\cdot\|_{p,\calY}$ denote the $L^p(\calY;\mu)$-norm. In particular, for $p<\infty$,
$$
\|\psi\|_{p,\calY} = \left(\int_\calY|\psi(y)|^p\,d\mu(y)\right)^{1/p}.
$$
Moreover, we define $\<\phi,\psi\>_{2,\calY} := \int_\calY\phi(y)\ol{\psi(y)}\,d\mu(y)$ for $\phi,\psi\in L^2(\calY;\mu)$. Similarly, we define the norm $\|\cdot\|_{p,\calX}$ on $L^p(\calX;\mu)$ and the scalar product $\<\cdot\,,\cdot\>_{2,\calX}$ on $L^2(\calX;\mu)$, respectively.

Let $B(\calY)$ denote the set of all bounded Borel-measurable functions on $\calY$. For $t\in [0,T]$, we define the {\em Koopman operator} $K^t$ on $B(\calY)$ associated with the SDE~\eqref{e:SDE0} by
$$
(K^t\psi)(x) = \int_\calY\psi(y)\,\rho_t(x,dy) = \bE\big[\psi(X_t)|X_0=x\big],\quad x\in\calX,
$$
for $\psi\in B(\calY)$.

\begin{rem}
(a) For the definition of $K^t$ to carry over to spaces of equivalence classes of functions that coincide a.e.\ on $\calY$, it must be ensured that $K^t\psi_1 = K^t\psi_2$ $\mu$-a.e.\ on $\calX$ whenever $\psi_1 = \psi_2$ $\mu$-a.e.\ on $\calY$. This is equivalent to the absolute continuity of the measure $\nu_t$ w.r.t.\ $\mu$ (see Remark \ref{r:on_assumptions}), which is a relaxation of \eqref{e:LipMeas}. However, it is easy to see that $K^t$ maps $L^1(\mu)$ (boundedly) into itself if and only if $\nu_t\le L\mu$ for some $L>0$, i.e., if \eqref{e:LipMeas} holds.

\smallskip\noindent
(b) Usually, the Koopman operator of a dynamical system on $\calX$ is considered as a linear operator from a space of functions on $\calX$ {\em into itself}. The price we pay by relaxing the invariance condition on $\calX$ is that our version of the Koopman operator is a linear map between {\em two different spaces}.

\smallskip\noindent
(c) Let us briefly consider the deterministic case. Here, we have $\rho_t(x,A) = \delta_{F_t(x)}(A)$, and thus $(K^t\psi)(x) = \int_\calY\psi(y)d\delta_{F_t(x)}(y) = \psi(F_t(x))$ for $t\in [0,T]$, cf.\ \eqref{e:stay_in_Y}. This complies with the usual definiton of the Koopman operator for deterministic dynamical systems as a composition operator.
\end{rem}

A proof of the following proposition can be found in the Appendix Section \ref{a:proof_Kt_bounded}.

\begin{prop}\label{p:Kbounded}
Let $p\in [1,\infty]$. For every $t\ge 0$, the operator $K^t$ extends uniquely to
\begin{itemize}
    \item a contraction from $L^\infty(\calY;\mu)$ to $L^p(\calX;\mu)$;
    \item a bounded operator from $L^p(\calY;\mu)$ to $L^p(\calX;\mu)$.
\end{itemize}
If $p<\infty$, for $\psi\in L^p(\calY;\mu)$ we have $K^t\psi\to\psi|_\calX$ in $L^p(\calX;\mu)$ as $t\to 0$.
\end{prop}

\begin{rem}
If $\calX$ is invariant under the SDE flow (i.e., $\calY = \calX$), the operators $K^t$ satisfy the semigroup property $K^{t+s} = K^tK^s$ for $s,t\ge 0$, which is a simple consequence of the Chapman-Kolmogorov equation \eqref{e:cke}. If $\mu$ is an invariant measure, $(K^t)_{t\ge 0}$ is a $C_0$-semigroup of contractions on\footnote{In what follows, $L(\calH,\calK)$ denotes the set of all bounded linear operators between Hilbert spaces $\calH$ and $\calK$. As usual, we also set $L(\calH) := L(\calH,\calH)$.} $L(L^p(\calX;\pi))$, where $p\in [1,\infty)$. However, if $\mu$ is not invariant but at least satisfies \eqref{e:LipMeas} on each interval $[0,T]$ (with the constants $L$ depending on $T$), the semigroup $(K^t)_{t\ge 0}$ might not be uniformly bounded. An example is given in Appendix \ref{s:unbounded}.
\end{rem}


\subsection{Reproducing kernel Hilbert spaces}\label{subsec:rkhs}
In what follows, let $k : \calY\times\calY\to\R$ be a continuous and bounded symmetric positive definite kernel, that is, we have $k(x,y) = k(y,x)$ for all $x,y\in \calY$ and
$$
\sum_{i,j=1}^m k(x_i,x_j)c_ic_j\,\ge\,0
$$
for all choices of $x_1,\ldots,x_m\in \calY$ and $c_1,\ldots,c_m\in\R$. It is well known that $k$ generates a so-called {\em reproducing kernel Hilbert space} (RKHS) \cite{a,bt,pr} $(\bH_\calY,\<\cdot\,,\cdot\>_{\bH_\calY})$ of bounded continuous functions on~$\calY$, such that for $\psi \in \bH_\calY$ the {\em reproducing property}
\begin{equation}\label{e:reproducing_property}
\psi(x) = \<\psi,\Phi_\calY(x)\>_{\bH_\calY},\qquad x\in \calY,
\end{equation}
holds, where $\Phi_\calY : \calY\to\bH_\calY$ denotes the so-called {\em feature map} corresponding to the kernel $k$, i.e.,
$$
\Phi_\calY(x) = k(x,\cdot),\qquad x\in \calY.
$$
In the sequel, we shall denote the norm on $\bH_\calY$ by $\|\cdot\|_{\bH_\calY}$ and the kernel diagonal by $\vphi$:
$$
\vphi(x) = k(x,x),\qquad x\in \calY.
$$
Then for $x\in \calY$ we have
$$
\|\Phi_\calY(x)\|_{\bH_\calY}^2 = \<\Phi_\calY(x),\Phi_\calY(x)\>_{\bH_\calY} = \<k(x,\cdot),k(x,\cdot)\>_{\bH_\calY} = k(x,x) = \vphi(x).
$$
We shall frequently make use of the following estimate:
$$
|k(x,y)| = |\<\Phi_\calY(x),\Phi_\calY(y)\>_{\bH_\calY}|\le\|\Phi_\calY(x)\|_{\bH_\calY}\|\Phi_\calY(y)\|_{\bH_\calY} = \sqrt{\vphi(x)\vphi(y)}.
$$
For example, it shows that $\|k\|_\infty = \|\vphi\|_\infty$.

In contrast to other works, where the kernel is induced by an observation map \cite{das2021} or is specifically designed by taking samples of the dynamics into account \cite{bblshh}, we admit any continuous kernel on $\calY$ which satisfies the following

\medskip\noindent
{\bf Compatibility assumptions:}

\begin{enumerate}
\item[(C1)] If $\psi\in L^2(\calY;\mu)$ such that $\int_\calY\!\int_\calY  k(x,y)\psi(x)\psi(y)\,d\mu(x)\,d\mu(y) = 0$, then $\psi=0$ $\mu$-a.e.
\item[(C2)] If $\psi\in\bH_\calY$ such that $\psi(y)=0$ for $\mu$-a.e.\ $y\in\calY$, then $\psi=0$.
\end{enumerate}

\smallskip
\noindent Since the functions in $\bH_\calY$ are continuous, note that condition (C2) holds if the measure $\mu$ is absolutely continuous with respect to Lebesgue measure with positive density. The condition allows us to consider $\bH_\calY$ as a subspace of $L^2(\calY;\mu)$, i.e., the inclusion operator $S_\calY : \bH_\calY\to L^2(\calY;\mu)$, defined by
\[
S_\calY\psi := [\psi],\quad\psi\in\bH_\calY,
\]
is injective. Here, $[\psi]$ denotes the equivalence class w.r.t.\ a.e.\ equality. It is easily seen that $S_\calY$ is bounded with norm $\|S_\calY\|\le \|\vphi\|_{1,\calY}^{1/2}$. Hence, $\bH_\calY$ is continuously embedded in $L^2(\calY;\mu)$, i.e.,
\[
\|\psi\|_{2,\calY}\lesssim\|\psi\|_{\bH_\calY},\quad\psi\in\bH_\calY.
\]
Furthermore, the condition {\rm (C1)} is equivalent to the density of $\bH_\calY$ in $L^2(\calY;\mu)$, see \cite{PhilScha23}. Hence, $\bH_\calY$ is continuously and densely embedded in $L^2(\calY;\mu)$.

In fact, the embedding is even compact as $S_\calY$ is a {\em Hilbert-Schmidt operator} with Hilbert-Schmidt norm $\|S_\calY\|_{HS}^2 = \|\vphi\|_{1,\calY}$, see \cite{PhilScha23}. For this, recall that an operator $T\in L(\calH)$ on a Hilbert space $\calH$ is {\em trace class} if for some (and hence for each) orthonormal basis (ONB) $(e_j)_{j\in\N}$ of $\calH$ we have $\sum_{j=1}^\infty\<(T^*T)^{1/2}e_i,e_i\> < \infty$. An operator $S\in L(\calH,\calK)$ between Hilbert spaces $\calH$ and $\calK$ is said to be {\em Hilbert-Schmidt} \cite[Chapter III.9]{gk} if $S^*S$ is trace class, i.e., $\|S\|_{HS}^2 := \sum_{j=1}^\infty\|Se_i\|^2 < \infty$ for some (and hence for each) ONB $(e_j)_{j\in\N}$ of $\calH$. Trace-class and Hilbert-Schmidt operators are compact operators. In particular, their spectrum consists of discrete eigenvalues with the origin as the only accumulation point. As $S_\calY$ is Hilbert-Schmidt, it follows that the {\em covariance operator}
$$
C_\calY :=S_\calY^*S_\calY\,\in\,L(\bH_\calY)
$$
is trace-class. Mercer's theorem shows the existence of a particular ONB $(e_j)_{j=1}^\infty$ of $L^2(\calY;\mu)$ composed of eigenfunctions of $S_\calY S_\calY^*$, which we shall henceforth call the {\em Mercer basis} corresponding to the kernel~$k$.

\begin{thm}[Mercer's Theorem \cite{rn}]\label{t:mercer}
There exists an ONB $(e_j)_{j=1}^\infty$ of $L^2(\calY;\mu)$ comprised of eigenfunctions of $S_\calY S_\calY^*$ with corresponding eigenvalues $\la_j > 0$ such that $\sum_{j=1}^\infty\la_j = \|\vphi\|_{1,\calY} < \infty$. Furthermore, $(f_j)_{j=1}^\infty$ with $f_j = \sqrt{\la_j}e_j$ constitutes an ONB of $\bH_\calY$ consisting of eigenfunctions of $C_\calY$ with corresponding eigenvalues $\la_j$. Moreover, for all $x,y\in \calY$,
$$
k(x,y) = \sum_jf_j(x)f_j(y) = \sum_j\la_je_j(x)e_j(y),
$$
where the series converges absolutely.
\end{thm}

The restriction $k_\calX$ of the kernel $k$ to $\calX\times\calX$ is certainly also a continuous and bounded symmetric positive definite kernel. Hence, the discussion above applies with $\calY$ replaced by $\calX$, which introduces the symbols $\bH_\calX$, $\Phi_\calX$, $S_\calX$, and $C_\calX$. Note that condition (C1) for $\calY$ implies (C1) for $\calX$. However, the analogous implication does not apply to (C2), which we explicitly assume here for $\calX$:
\begin{enumerate}
\item[(C2$_\calX$)] If $\psi\in\bH_\calX$ such that $\psi(x)=0$ for $\mu$-a.e.\ $x\in\calX$, then $\psi=0$ on $\calX$.
\end{enumerate}
Then all the statements above also hold for $\calY$ replaced by $\calX$. Again, (C2$_\calX$) holds if the measure $\mu$ is absolutely continuous with respect to Lebesgue measure with positive density. By \cite[Corollary 5.8]{pr},
$$
\bH_\calX = \{\psi|_\calX : \psi\in\bH_\calY\}
$$
and
$$
\|\psi\|_{\bH_\calX} = \inf\{\|\eta\|_{\bH_\calY} : \eta|_\calX = \psi\}.
$$

\smallskip
\noindent{\bf Covariance and cross-covariance operators.} For any $x\in\calX$ and $y\in\calY$ define the rank-one operator $\Phi(x)\otimes\Phi(y) : \bH_\calY\to\bH_\calX$ by
\[
[\Phi(x)\otimes\Phi(y)]\,\psi := \<\psi,\Phi(y)\>_{\bH_\calY}\Phi(x) = \psi(y)\Phi(x),\quad \psi\in\bH_\calY.
\]
Clearly, if $x\in\calX$, then $\Phi(x)\otimes\Phi(x) : \bH_\calX\to\bH_\calX$. The {\em covariance operator} $C_\calX = S_\calX^*S_\calX$ then has the representation
$$
C_\calX = \int_\calX \Phi(x)\otimes\Phi(x)\,d\mu(x).
$$
For $t\ge 0$ and $\psi\in\bH_\calY$ we further define the {\em cross-covariance operator} $C_{\calX\calY}^t : \bH_\calY\to\bH_\calX$ by
$$
C_{\calX\calY}^t := \int_\calX\int_\calY \Phi(x)\otimes\Phi(y)\,\rho_t(x,dy)\,d\mu(x) = S_\calX^*K^t S_\calY.
$$
As the product of the two Hilbert-Schmidt operators $S_\calX^*K^t$ and $S_\calY$, the operator $C_{\calX\calY}^t$ is trace class for all $t\ge 0$ (cf.\ \cite[p.\ 521]{k}). Note that for $t=0$ we have $C_{\calX\calY}^t\psi = S_\calX^*S_\calX\psi|_\calX = C_\calX\psi|_\calX$. Moreover, for all $\eta\in\bH_\calX$ and $\psi\in\bH_\calY$ we have
\begin{align}\label{e:umrechnen}
\<\eta,C_{\calX\calY}^t\psi\>_{\bH_\calX} = \<\eta,K^t\psi\>_{2,\calX}.
\end{align}

\subsection{Data samples and Kernel EDMD}\label{subsec:kEDMD}
Next, we briefly recall the Koopman approximation method called {\em Extended Dynamic Mode Decomposition} (EDMD), where certain (usually continuous) functions from $L^2(\calY;\mu)$ are evaluated on a set of sampled data.

\begin{ass}
The data $(x_k,y_k)$, $k=0,\ldots,m-1$, are drawn i.i.d.\ from the joint distribution
$$
d\mu_{0,t}(x,y) := \rho_t(x,dy)\,d\mu(x)
$$
on $\calX\times\calY$, where the time-step $t>0$ is fixed. This then implies that the $x_k$ are sampled i.i.d.\ with respect to $\mu$ and $y_k$ is drawn dependently of $x_k$ from the probability distribution $\rho_t(x_k,\,\cdot\,)$.
\end{ass}

\noindent{\bf The EDMD Algorithm \cite{williams15}.} In EDMD, one chooses functions in $L^2(\calY;\mu)$ from a so-called {\em dictionary} $\calD = \{\psi_1,\ldots,\psi_N\}$ spanning the finite-dimensional space $\bV = \linspan\{\psi_1,\ldots,\psi_N\}\subset L^2(\calY;\mu)$ with the goal of approximating the compression $P_\bV K^t|_\bV$ of $K^t$ to $\bV$. Here, $P_\bV$ denotes the orthogonal projection onto the space $\bV$ in $L^2(\calY;\mu)$. The matrix representation of the compression operator with respect to the basis $\calD$ of $\bV$ is then given by $C^{-1}A$, where $C = (\<\psi_i,\psi_j\>_\mu)_{i,j=1}^N$ is the so-called mass matrix and $A = (\<\psi_i,K^t\psi_j\>_\mu)_{i,j=1}^N$ is called the stiffness matrix. In order to define empirical estimators for $C$ and $A$, we replace the scalar products in the matrices' entries by sums of $\psi_i(x_k)\psi_j(x_k)$ and $\psi_i(x_k)\psi_j(y_k)$, respectively, which leads to the matrices
\[
\wh C_m = \tfrac 1m\Psi(X)\Psi(X)^\top
\qquad\text{and}\qquad
\wh A_m = \tfrac 1m\Psi(X)\Psi(Y)^\top,
\]
where $\Psi = [\psi_1,\ldots,\psi_N]^\top$ and
\[
\Psi(X) = [\Psi(x_0),\ldots,\Psi(x_{m-1})]
\qquad\text{and}\qquad
\Psi(Y) = [\Psi(y_0),\ldots,\Psi(y_{m-1})].
\]
Hence, the EDMD approximation of $P_\bV K^t|_\bV$ is given by the estimator $\wh M^{m,t} = \wh C_m^{-1}\wh A_m$. In \cite{nueske23} it was proved that $\wh M^{m,t}$ indeed converges in probability to $C^{-1}A$ as $m\to\infty$, and bounds on the accuracy were provided.

\smallskip
\noindent{\bf The kEDMD Algorithm \cite{williams15_kernel,klus20}.} In contrast to traditional EDMD, in {\em kernel EDMD} (kEDMD) the dictionary is not chosen a priori, but is given by the features sampled on the data points, that is, $\calD = \{\Phi(x_0),\ldots,\Phi(x_{m-1})\}$. The data matrices $\Psi(X)$ and $\Psi(Y)$ are then easily seen to be given by
\[
\Psi(X) = \big(k(x_i,x_j)\big)_{i,j=0}^{m-1}
\qquad\text{and}\qquad
\Psi(Y) = \big(k(x_i,y_j)\big)_{i,j=0}^{m-1}.
\]
Hence, since $\Psi(X)$ is symmetric,
\[
\wh C_m^{-1}\wh A_m = \big[\tfrac 1m\Psi(X)\Psi(X)^\top\big]^{-1}\big[\tfrac 1m\Psi(X)\Psi(Y)^\top\big] = \Psi(X)^{-1}\Psi(Y)^\top.
\]
However, since the eigenvalues of $\Psi(X)$ typically decay fast, in practice one often replaces the inverse of the ill-conditioned matrix $\Psi(X)$ by a pseudo-inverse of a rank-$r$ ($r<m$) eigendecomposition truncation of $\Psi(X)$. Hence, the kEDMD estimator for the Koopman operator compression becomes
\begin{align}\label{e:kedmd_est}
\wh M^{m,t}_r = [\Psi(X)]_r^\dagger\Psi(Y)^\top.    
\end{align}
Now, the 
empirical estimators for the operators $C_\calX$ and $C_{\calX\calY}^t$ are given by
\[
\wh C_\calX^m = \frac 1m\sum_{k=0}^{m-1}\Phi(x_k)\otimes\Phi(x_k)
\qquad\text{and}\qquad
\wh C_{\calX\calY}^{m,t} = \frac 1m\sum_{k=0}^{m-1}\Phi(x_k)\otimes\Phi(y_k).
\]
We have $\wh C_\calX^m : \bH_\calX\to\bH_\calX$ and $\wh C_{\calX\calY}^{m,t} : \bH_\calY\to\bH_\calX$. In fact, both operators map to $\bV = \linspan\calD\subset\bH_\calX$. Since
\begin{align*}
\wh C_\calX^m\Phi(x_j) &= \frac 1m\sum_{k=0}^{m-1}[\Phi(x_k)\otimes\Phi(x_k)]\Phi(x_j) = \frac 1m\sum_{k=0}^{m-1}k(x_j,x_k)\Phi(x_k),\\
\wh C_{\calX\calY}^{m,t}\Phi(x_j) &= \frac 1m\sum_{k=0}^{m-1}[\Phi(x_k)\otimes\Phi(y_k)]\Phi(x_j) = \frac 1m\sum_{k=0}^{m-1}k(x_j,y_k)\Phi(x_k),
\end{align*}
the matrix representations of $\wh C_\calX^m|_\bV$ and $\wh C_{\calX\calY}^{m,t}|_\bV$ with respect to the basis $\calD = \{\Phi(x_0),\ldots,\Phi(x_{m-1})\}$ are given by $\frac 1m\Psi(X)$ and $\frac 1m\Psi(Y)^\top$, respectively. Thus, the kEDMD estimator matrix $\wh M^{m,t}_r$ in \eqref{e:kedmd_est} represents the operator $[\wh C_\calX^m]_r^\dagger\wh C_{\calX\calY}^{m,t}|_\bV$. In the next section, we shall give probabilistic bounds on the $\|\cdot\|_{\bH_\calY\to L^2(\calX;\mu)}$-error between $P_r K^t$ and
\begin{equation}\label{e:Koopman_est}
\wh K_r^{m,t} := [\wh C_\calX^m]_r^\dagger \wh C_{\calX\calY}^{m,t},
\end{equation}
which then serves as a quantification of the prediction error made by kEDMD. Here, $P_r$ denotes the orthogonal projection in $L^2(\calX;\mu)$ onto the span of eigenvectors corresponding to the largest $r$ eigenvalues of $C_\calX$.

\medskip
\section{Bound on the Koopman prediction error}\label{sec:res1}
Before we present our main result, let us bound the error made by estimating the operators $C_\calX$ and $C_{\calX\calY}$ by their empirical estimators $\wh C_\calX^m$ and $\wh C_{\calX\calY}^{m,t}$. Since $C_\calX$ and $C_{\calX\calY}$ are Hilbert-Schmidt, we may quantify the error in the Hilbert-Schmidt norm.

\begin{prop}\label{p:prob_est}
For every $t\ge 0$, the following probabilistic bounds on the Hilbert-Schmidt estimation errors holds:
\begin{align*}
\bP\big(\|C_{\calX\calY}^t - \wh C_{\calX\calY}^{m,t}\|_{HS} > \veps\big)
&\le 2 \,e^{-\frac{m\veps^2}{8\|k\|_\infty^2}}\\
\bP\big(\|C_{\calX} - \wh C_{\calX}^{m}\|_{HS} > \veps\big)
&\le 2 \,e^{-\frac{m\veps^2}{8\|k\|_\infty^2}}.
\end{align*}
\end{prop}
\begin{proof}
For $x\in\calX$ and $y\in\calY$ let us abbreviate $C_x := \Phi(x)\otimes\Phi(x)$ and $C_{xy} := \Phi(x)\otimes\Phi(y)$. If $(f_i)\subset\bH_\calY$ denotes the Mercer ONB of $\bH_\calY$ corresponding to $k$ on $\calY\times\calY$, for $x,x'\in\calX$ and $y,y'\in\calY$ we have
$$
\<C_{xy},C_{x'y'}\>_{HS} = \sum_i\<C_{xy}f_i,C_{x'y'}f_i\>_{\bH_\calX} = \sum_if_i(y)f_i(y')k(x,x') = k(x,x')k(y,y').
$$
This proves $\|C_{xy}\|_{HS}^2 = \vphi(x)\vphi(y)$. Moreover, it yields
\begin{align*}
\|C_{\calX\calY}^t\|_{HS}^2
&= \left\|\int C_{xy}\,d\mu_{0,t}(x,y)\right\|_{HS}^2 = \int\int k(x,x')k(y,y')\,d\mu_{0,t}(x,y)\,d\mu_{0,t}(x',y')\,\le\,\|k\|_\infty^2.
\end{align*}
Hence, since the $C_{x_k,y_k}$ are independent, $\bE[C_{\calX\calY}^t - C_{x_k,y_k}]=0$, and
$$
\|C_{\calX\calY}^t - C_{x_k,y_k}\|_{HS}\le \|C_{\calX\calY}^t\|_{HS} + \|C_{x_k,y_k}\|_{HS}\le 2\|k\|_\infty,
$$
we may apply Hoeffding's inequality for Hilbert space-valued random variables \cite[Theorem 3.5]{pi} (see also \cite[Theorem A.5.2]{m}) to $C_{\calX\calY}^t - C_{x_k,y_k}$ and obtain
$$
\bP\big(\|C_{\calX\calY}^t - \wh C_{\calX\calY}^{m,t}\|_{HS} > \veps\big) = \bP\left(\left\|\frac 1m\sum_{k=0}^{m-1}(C_{\calX\calY}^t - C_{x_k,y_k})\right\|_{HS} > \veps\right)\,\le\,2e^{-\frac{m\veps^2}{8\|k\|_\infty^2}}.
$$
The proof of the second claim follows analogous lines.
\end{proof}

Let $(e_j)$ be the Mercer orthonormal basis of $L^2(\calX;\mu)$ corresponding to the kernel $k$ and let $\la_j = \|S_\calX S_\calX^*e_j\|_{2,\calX}$ as well as $f_j := \sqrt{\la_j}e_j\in\bH_\calX$ (cf.\ Theorem \ref{t:mercer}). We arrange the Mercer eigenvalues $\la_j$ of $C_\calX$ in a non-increasing way, i.e.,
$$
\la_1\ge\la_2\ge\ldots
$$
and denote by $P_r$ the orthogonal projection in $L^2(\calX;\mu)$ onto $\linspan\{e_1,\ldots,e_r\}$.

\begin{thm}\label{t:main}
Fix an arbitrary $r\in\N$ and assume that the first $r+1$ eigenvalues $\la_j$ of $C_\calX$ are simple, i.e., $\la_{j+1} < \la_j$ for all $j=1,\ldots,r$, and let
\begin{equation}\label{e:delta_r}
\delta_r = \min_{j=1,\ldots,r}\frac{\la_j - \la_{j+1}}{2}
\qquad\text{and}\quad c_r = \frac 1{\sqrt{\la_r}} + \frac{r+1}{\delta_r\la_r}(1+\|\vphi\|_{1,\calY})\|\vphi\|_{1,\calX}^{1/2}.
\end{equation}
Further, let $\veps\in (0,\delta_r)$ and $\delta\in (0,1)$ be arbitrary, fix some $m\ge\max\{r,\frac{8\|k\|_\infty^2\ln(4/\delta)}{\varepsilon^2}\}$ and define $\wh K_r^{m,t}$ as in \eqref{e:Koopman_est}. Then with probability at least $1-\delta$, we have that
\begin{equation*}
\|P_rK^t - \wh K_r^{m,t}\|_{\bH_\calY\to L^2(\calX;\mu)}\,\le\,c_r\veps.
\end{equation*}
If, furthermore, $K^t\bH_\calY\subset\bH_\calX$, then, with probability at least $1-\delta$,
\begin{equation}\label{e:full_est}
\|K^t - \wh K_r^{m,t}\|_{\bH_\calY\to L^2(\calX;\mu)}\,\le\,\sqrt{\la_{r+1}}\,\|K^t\|_{\bH_\calY\to\bH_\calX} + c_r\veps.
\end{equation}
\end{thm}
\begin{proof}[Proof of Theorem \rmref{t:main}]
Let $\wh\la_1\ge\ldots\ge\wh\la_r$ denote the largest $r$ eigenvalues of $\wh C_\calX^m$ in descending order and let $\wh e_1,\ldots,\wh e_r$ be corresponding eigenfunctions, respectively, such that $\|\wh e_j\|_{\bH_\calX} = \wh\la_j^{-1/2}$ for $j=1,\ldots,m$. If we set $\wh f_j = \wh\la_j^{1/2}\cdot\wh e_j$, then for $\psi\in\bH_\calY$ we have
\[
\wh K_r^{m,t}\psi = [\wh C_\calX^m]_r^\dagger\wh C_{\calX\calY}^{m,t}\psi = \sum_{j=1}^r\wh\la_j^{-1}\<\wh C_{\calX\calY}^{m,t}\psi,\wh f_j\>_{\bH_\calX}\wh f_j = \sum_{j=1}^r\<\wh C_{\calX\calY}^{m,t}\psi,\wh e_j\>_{\bH_\calX}\wh e_j,
\]
which estimates (see \eqref{e:umrechnen})
\[
\sum_{j=1}^r\<C_{\calX\calY}^t\psi,e_j\>_{\bH_\calX}e_j = \sum_{j=1}^r\<K^t\psi,e_j\>_{2,\calX}e_j = P_rK^t\psi.
\]
Hence, we have to estimate the deviation of $\sum_{j=1}^r\<\wh C_{\calX\calY}^{m,t}\psi,\wh e_j\>_{\bH_\calX}\wh e_j$ from $\sum_{j=1}^r\<C_{\calX\calY}^t\psi,e_j\>_{\bH_\calX}e_j$, which was done in \cite{PhilScha23} for the special case of $\calX = \calY = \R^d$. The proof for the present case follows analogous lines. Also the estimate in the case $K^t\bH_\calY\subset\bH_\calX$ can be proved by following the lines of the proof of the counterpart in \cite{PhilScha23}.
\end{proof}

\begin{rem}
The assumption $K^t\bH_\calY\subset\bH_\calX$ is not too exotic. In fact, we prove in Section \ref{a:OU_invariance} in the Appendix that the condition holds for the Ornstein-Uhlenbeck process, if $\calY = \R$, $\calX\subset\R$ is any compact set with non-empty interior, and $k$ is a Gaussian RBF kernel.
\end{rem}

\section{Extension to deterministic control-affine systems}\label{s:control}
\noindent In this part, we extend the error estimates to deterministic control systems of the form
\begin{align}\label{eq:controlsystem}
\dot{x}(t) = f(x(t)) + \sum_{i=1}^{\nc} u_i(t) g_i(x(t)),\qquad t\in [0,T],
\end{align}
with vector fields $f:\R^d \to \R^d$ and $g_i:\R^d \to \R^d$, $i\in [1:\nc] := \{1,\ldots,\nc\}$, $\nc\in\N$, and a prediction horizon $T>0$. We let $\calX$ and $\calY$ be compact sets in $\R^d$ and choose specifically
$$
\mathrm{d}\mu(x) = \frac 1{|\calX|}\,\mathrm{d}x,
$$
where $|\calX|$ denotes the Lebesgue measure of $\calX$ on $\R^d$. We furthermore assume that the restrictions of the vector fields $f$ and $g_1,\ldots,g_\nc$ to $\calY$ are $C^1$. Admissible controls $u$ in \eqref{eq:controlsystem} are measurable functions $u : [0,T]\to\calU$, where $\calU\subset\R^\nc$ is a compact, convex set satisfying $0\in\operatorname{int}(\calU)$. To avoid excessive notation, we assume that  for all $x^0\in\calX$ and all admissible controls $u : [0,T]\to\calU$,
\begin{itemize}
\item $[0,T]$ is contained in the existence interval for the solution $x(\,\cdot\,;x^0,u)$ to \eqref{eq:controlsystem} and
\item $x(t;x^0,u)\in\calY$ for each $t\in [0,T]$.
\end{itemize}
If the first assumption is satisfied and $\calX$ is contained in the interior of $\calY$, then by continuity of the flow in time and the compactness of both $\calX$ and $\calU$, there exists a time horizon $T_0$ such that the second assumption is satisfied for all $T\in (0,T_0)$.

Here, we consider sampled-data systems with Zero-order Hold (ZoH), i.e., for a given, typically small sampling period~$t>0$, $u(s) \equiv u$ on $[kt,(k+1)t)$, $k \in \mathbb{N}_0$. This system class is well established in systems and control, see, e.g., \cite{CastGenn97,NesiTeel07} and the references therein. 
Sampled-data systems 
can be used, e.g., for output-feedback stabilization of nonlinear systems~\cite{LiZhao18} or in the redesign of continuous-time control signals~\cite{GrunWort08}. In particular, there are many results showing that system-theoretic properties like stabilizability, or even cost controllability, are preserved under sampling~\cite{WortRebl14}. The latter even 
is inherited by the EDMD-based surrogate model in the Koopman framework as recently shown in~\cite{BoldGrun23} analyzing EDMD-based Model Predictive Control.
Hence, we also restrict our consideration to sampled-data systems knowing that the analysis on one sampling period suffices to obtain similar results based on the concatenation of the control signal for multiple sampling periods.

For a constant control $u\in\calU$ we let $K_u^t$ be the Koopman operator for the autonomous system \eqref{eq:controlsystem}, that is,
\begin{align*}
(K_u^t \psi)(x^0) := \psi\big(x(t;x^0,u)\big),\qquad x^0\in\calX,\,\psi\in L^2(\calY).
\end{align*}
In order to control the system by piecewise constant controls, in principle, one could compute the kEDMD estimator for $K_u^t$ for every applied control $u$, which, however, would require a lot of computation time in each step. As this is infeasible in practice, we circumvent this issue by approximating $K_u^t$ by a linear combination of $\nc+1$ Koopman operators $K_{u_i}^t$ with quadratic precision. In this course, we leverage the affine representation of the Koopman generator.

The Koopman generator $\calL^u$ associated with \eqref{eq:controlsystem} is given by
\begin{align}\label{eq:gen_to_op}
\calL^u \psi = \lim_{t\to 0} \frac{K_u^t \psi - \psi|_\calX}{t}
\end{align}
for $\psi\in\dom\calL^u$, where the limit is taken in $L^2(\calX)$, and $\dom\calL^u$ consists of those functions $\psi\in L^2(\calY)$ for which this $L^2$-limit exists. It is easy to see that $C^1(\calY)\subset\dom\calL^u$ for each $u\in\calU$ such that for $\psi\in C^1(\calY)$ and $x^0\in\calX$,
\begin{align}
(\calL^u \psi)(x^0)
&= \notag\lim_{t\to 0}\tfrac 1t\big((K_u^t\psi)(x^0) - \psi(x^0)\big) = \frac{\mathrm{d}}{\mathrm{d}t}(K_u^t\psi)(x^0)\Big|_{t=0} = \frac{\mathrm{d}}{\mathrm{d}t}\psi(x(t;x^0,u))\Big\vert_{t=0}\\
&= \nabla \psi(x^0)^\top\Big(f(x^0) + \sum_{i=1}^{\nc} u_i g_i(x^0)\Big)\label{e:with_nabla}.
\end{align}
Let $\{e_i\}_{i=1}^{\nc}$ denote the standard basis of $\R^\nc$ and let 
\[
\gamma_i := \inf\{\gamma > 0 : \gamma e_i\notin\calU\},\qquad i = 1,\ldots,\nc.
\]
Then the $\gamma_i$ are positive numbers thanks to $0\in\operatorname{int}(\calU)$, and we have $\gamma_ie_i\in\partial\calU\subset\calU$. Now, if $u = \sum_{i=1}^\nc u_ie_i$ is the coordinate representation of the vector $u\in\calU$, the representation \eqref{e:with_nabla} reveals that, on $C^1(\calY)$,
\begin{align}\label{eq:generatoraffine}
\calL^u = \calL^0 + \sum_{i=1}^{\nc} \frac{u_i}{\gamma_i}\left(\calL^{\gamma_i e_i}-\calL^0\right).
\end{align}
To transfer the error bounds from autonomous systems derived in this work to control systems, we will show that the restriction of the Koopman operator to $C^2(\calY)$ is also affine in the control---at least approximately for small $t>0$. 
For this, set
$$
f_u(x) := f(x)+ \sum_{i=1}^{\nc}u_i g_i(x)
$$
and let $\psi \in C^2 (\calY)$, $u\in \calU$, and $x^0\in\calX$. We apply Taylor's theorem to $\psi\circ x(\,\cdot\,;x^0,u)$ at $t_0=0$ and observe that
\begin{align*}
(K_{u}^t\psi - \psi|_\calX)(x^0) &= \psi(x(t;x^0,u)) - \psi(x(0;x^0,u)) \\&= t\nabla\psi(x^0)\cdot f_u(x^0) + \frac{1}{2}r(\xi)t^2,
\end{align*}
for some $\xi\in [0,t]$, where, abbreviating $x(s) = x(s;x^0,u)$,
\begin{align*}
r(s) = f_u(x(s))^\top\nabla^2 \psi(x(s))f_u(x(s)) + \nabla \psi(x(s))^\top Df_u(x(s))f_u(x(s)).
\end{align*}
By compactness of $\calU$, $\calX$, and $\calY$ and the continuity of all functions involved, we can bound the remainder term uniformly in $u\in\calU$, $x^0\in\calX$ and $s\in [0,T]$. More specifically,
\begin{align}\label{e:est_remainder}
\max_{s\in [0,T]}|r(s)| \leq c(f,g_1,\ldots,g_{n_u})\cdot\|\psi\|_{C^2(\calY)},
\end{align}
where $\|\psi\|_{C^2(\calY)}:= \|\psi\|_{C(\calY)} + \|\nabla\psi\|_{C(\calY)^d} + \|\nabla^2\psi\|_{C(\calY)^{d\times d}}$. Thus, 
\begin{align}\label{eq:approxkoop}
K_{u}^t\psi =  \psi + t\calL^{u}\psi + \mathcal{O}(t^2)
\end{align}
where the constant in $\calO(t^2)$ as defined in \eqref{e:est_remainder} can be chosen independently of $u$.

Invoking control-affinity of the generator~\eqref{eq:generatoraffine}, we thus get
\begin{align}\label{eq:operator_controlaffine}
K_u^t\psi = K^t \psi + \sum_{i=1}^{\nc} \frac{u_i}{\gamma_i}\left(K^t_{\gamma_i e_i}-K^t\right)\psi + O(t^2),
\end{align}
which means that the Koopman operator is control affine up to second order in time. This is Equation~(3.14) in \cite{PeitOtto2020}, where we stress however that the remainder term here depends on the chosen dictionary function.

The error bound~\eqref{e:full_est} obtained in Theorem~\ref{t:main} can now be transferred to the controlled case. Thus, for $u\in\calU$ we define the estimator
\begin{align*}
\wh K^{m,t}_{r,u} := \wh K^{m,t}_{r} + \sum_{i=1}^{\nc} \frac{u_i}{\gamma_i}\big(\wh{K}^{m,t}_{r,\gamma_i e_i} - \wh{K}^{m,t}_{r}\big),
\end{align*}
where the $\wh K_{r,\gamma_ie_i}^{m,t}$ are defined as in \eqref{e:Koopman_est} for the system $\dot x = f(x) + \gamma_ig_i(x)$.

\begin{thm}\label{t:control}
Let the assumptions in Theorem \ref{t:main} be satisfied. For fixed $r\in\N$, let $\delta_r$ and $c_r$ be as in~\eqref{e:delta_r}. Further, let $\veps\in (0,\delta_r)$ and $\delta\in (0,1)$ be arbitrary and fix some $m\ge\max\{r,\frac{8\|k\|_\infty^2\ln(4(n_u+1)/\delta)}{\varepsilon^2}\}$. Then, for each $\psi\in\bH_\calY\cap C^2(\calX)$, there is a constant $c_\psi\geq 0$ such that with probability at least $1-\delta$ for all $u\in\calU$ we have
\begin{align}\label{e:control_est}
\|P_rK_u^t\psi - \wh K_{r,u}^{m,t}\psi\|_{L^2(\calX;\mu)}\,\le\,(1+2\|\Gamma^{-1}u\|_1)c_r\veps\|\psi\|_{\bH_\calY} + c_\psi t^2.
\end{align}
where $\Gamma = \diag_{i=1,\ldots,n_u}\gamma_i$. In the case where $K_u^t\bH_\calY\subset\bH_\calX$ for all $u\in\calU_\partial := \{0,\gamma_1e_1,\ldots,\gamma_{n_u}e_{n_u}\}$, we obtain
\begin{align}\label{e:control_est2}
\|K_u^t\psi - \wh K_{r,u}^{m,t}\psi\|_{L^2(\calX;\mu)}\,\le\,(1+2\|\Gamma^{-1}u\|_1)(\sqrt{\la_{r+1}}M + c_r\veps)\|\psi\|_{\bH_\calY} + c_\psi t^2,
\end{align}
where $M = \max\{\|K_u^t\|_{\bH_\calY\to\bH_\calX} : u\in\calU_\partial\}$.

If the kernel $k$ enjoys $C^2$-regularity, for $m\ge\max\{r,\frac{8\|k\|_\infty^2\ln(4/\delta)}{\varepsilon^2}\}$ data points we have with probability at least $1-\delta$ that
\[
\|P_rK_u^t - \wh K_{r,u}^{m,t}\|_{\bH_\calY\to L^2(\calX;\mu)}\,\le\,c_r\veps + t^2\cdot C_k\wh\la_r^{-1}
\]
with a constant $C_k$ depending only on the kernel $k$. Here, $\wh\la_r$ denotes the $r$-th largest eigenvalue of the operator $\wh C_\calX^m$. Again, if $K_u^t\bH_\calY\subset\bH_\calX$ for all $u\in\calU_\partial$, then
\[
\|K_u^t - \wh K_{r,u}^{m,t}\|_{\bH_\calY\to L^2(\calX;\mu)}\,\le\,\sqrt{\la_{r+1}}M + c_r\veps + t^2\cdot C_k\wh\la_r^{-1}.
\]
\end{thm}
\begin{proof}
Using the approximate control-affinity \eqref{eq:operator_controlaffine}, we compute
\begin{align*}
(P_rK_u^t - \wh K_{r,u}^{m,t})\psi = (P_rK^t - \wh K_r^{m,t})\psi + \sum_{i=1}^\nc \frac{u_i}{\gamma_i}\left(P_rK^t_{\gamma_ie_i} - \wh K^{m,t}_{r,\gamma_ie_i} - P_rK^t + \wh K^{m,t}_{r}\right)\psi + t^2\phi,
\end{align*}
where $\|\phi\|_{L^\infty(\calX)}\le c_\psi$ for all choices of $u\in\calU$ and $t\in [0,T]$. By the choice of $m$, the respective bounds in Theorem~\ref{t:main} on $\|P_rK^t - \wh K_r^{m,t}\|$ and $\|P_rK^t_{\gamma_ie_i} - \wh K^{m,t}_{r,\gamma_ie_i}\|$, $i\in\{1,\ldots,\nc\}$, hold with probability at least $1-\delta(n_u+1)^{-1}$ each. Hence, the probability for all $n_u+1$ bounds to hold exceeds $1-\delta$, cf.\ \cite[Lemma 22]{nueske23}. In that event,
\begin{align*}
&\Big\|(P_rK^t - \wh K_r^{m,t})\psi + \sum_{i=1}^\nc \frac{u_i}{\gamma_i}\left(P_rK^t_{\gamma_ie_i} - \wh K^{m,t}_{r,\gamma_ie_i} - P_rK^t + \wh K^{m,t}_{r}\right)\psi\Big\|_{L^2(\calX)}\\
&\le \|P_rK^t - \wh K_r^{m,t}\|_{\bH_\calY\to L^2(\calX)} \|\psi\|_{\bH_\calY}\\
&\qquad + \Big(\max_{i\in\{1,\ldots,\nc\}}\big\|P_rK^t_{\gamma_ie_i} - \wh K^{m,t}_{r,\gamma_ie_i}\big\|_{\bH_\calY\to L^2(\calX)} + \|P_rK^t - \wh K_r^{m,t}\|_{\bH_\calY\to L^2(\calX)}\Big) \|\psi\|_{\bH_\calY}  \sum_{i=1}^\nc \frac{|u_i|}{\gamma_i}\\
&\le \big(1 + 2\|\Gamma^{-1}u\|_1\big)c_r\veps\|\psi\|_{\bH_\calY},
\end{align*}
which proves the given estimate. The proof for the case $K_u^t\bH_\calY\subset\bH_\calX$ for all $u\in\calU_\partial$ follows analogous lines, making use of \eqref{e:full_est}.

Now, let $k\in C^2$ and denote the flow of $\dot x = f_v(x)$ by $F_v$, $v\in\calU$. Also, let $F := F_0$. In addition to the random variables $y_k = F^t(x_k)$, $k=0,\ldots,m-1$, let $y_{k,v} = F_v^t(x_k)$, $i=1,\ldots,\nc$. Moreover, let $\wh C^{m,t}_{\calX\calY,v}$ be the empirical estimator of the cross-covariance operator $C^t_{\calX\calY,v}$ for the system $\dot x = f_v(x)$. Then we have
\begin{align*}
\wh K^{m,t}_{r,u}
&= \wh K^{m,t}_{r} + \sum_{i=1}^{\nc} \frac{u_i}{\gamma_i}\big(\wh{K}^{m,t}_{r,\gamma_ie_i} - \wh{K}^{m,t}_{r}\big) = [\wh C_\calX^m]_r^\dagger\Big[ \wh C^{m,t}_{\calX\calY} + \sum_{i=1}^{\nc} \frac{u_i}{\gamma_i}\big(\wh C^{m,t}_{\calX\calY,\gamma_ie_i} - \wh C^{m,t}_{\calX\calY}\big) \Big]\\
&= [\wh C_\calX^m]_r^\dagger\,\frac 1m\sum_{k=0}^{m-1}\Phi(x_k)\otimes\Bigg[ \Phi(F^t(x_k)) + \sum_{i=1}^{\nc} \frac{u_i}{\gamma_i}\big(\Phi(F_{\gamma_ie_i}^t(x_k)) - \Phi(F^t(x_k))\big) \Bigg]\\
&= [\wh C_\calX^m]_r^\dagger\,\frac 1m\sum_{k=0}^{m-1}\Phi(x_k)\otimes\Bigg[ \bK^t\Phi(x_k) + \sum_{i=1}^{\nc} \frac{u_i}{\gamma_i}\big(\bK_{\gamma_ie_i}^t\Phi(x_k) - \bK^t\Phi(x_k)\big) \Bigg],
\end{align*}
where $\bK_v^t : L^2(\calY,\bH_\calY)\to L^2(\calX,\bH_\calY)$ for $v\in\calU$ is defined by $\bK_v^t\Psi(x) = \Psi(F_v^t(x))$, $\Psi\in L^2(\calY,\bH_\calY)$, and $\bK^t := \bK^t_0$. As $\Phi\in C^2$, we can, analogously as in the discussion preceding the theorem, see that
\[
\bK^t\Phi + \sum_{i=1}^{\nc}\frac{u_i}{\gamma_i}\left(\bK^t_{\gamma_ie_i}-\bK^t\right)\Phi = \bK_u^t\Phi + t^2\Psi_{u,t}
\]
with some $\Psi_{u,t}\in L^2(\calX,\bH_\calY)$ such that $C_\Phi = \sup\{\|\Psi_{u,t}\|_{L^\infty(\calX,\bH_\calY)} : u\in\calU,\,t\in [0,T]\} < \infty$. Hence, we obtain
\begin{align*}
\wh K^{m,t}_{r,u}
&= [\wh C_\calX^m]_r^\dagger\,\frac 1m\sum_{k=0}^{m-1}\Phi(x_k)\otimes\big[\Phi(F_u^t(x_k)) + t^2\Psi_{u,t}(x_k)\big]\\
&= [\wh C_\calX^m]_r^\dagger\Big[\wh C^{m,t}_{\calX\calY,u} + \frac{t^2}m\sum_{k=0}^{m-1}\Phi(x_k)\otimes\Psi_{u,t}(x_k)\Big].
\end{align*}
Note that $[\wh C_\calX^m]_r^\dagger\wh C^{m,t}_{\calX\calY,u}$ is the estimator for $K^t_u$ as defined in \eqref{e:Koopman_est}. Here, we shall denote it by $\wt K^{m,t}_{r,u}$. Now, if $\psi\in\bH_\calY$, then
\begin{align*}
\|P_rK^t\psi - \wh K^{m,t}_{r,u}\psi\|_{L^2(\calX)}
&\le \|P_rK^t\psi - \wt K^{m,t}_{r,u}\psi\|_{L^2(\calX)} + \|\wt K^{m,t}_{r,u}\psi - \wh K^{m,t}_{r,u}\psi\|_{L^2(\calX)}\\
&\le c_r\veps\|\psi\|_{\bH_\calY} + \frac{t^2}m\Bigg\|\sum_{k=0}^{m-1}[\wh C_\calX^m]_r^\dagger[\Phi(x_k)\otimes\Psi_{u,t}(x_k)]\psi\Bigg\|_{L^2(\calX)}\\
&\le c_r\veps\|\psi\|_{\bH_\calY} + \frac{t^2}m\sum_{k=0}^{m-1}|\<\psi,\Psi_{u,t}(x_k)\>_{\bH_\calY}|\cdot\|[\wh C_\calX^m]_r^\dagger\Phi(x_k)\|_{L^2(\calX)}\\
&\le c_r\veps\|\psi\|_{\bH_\calY} + \|\vphi\|_{1,\calX}^{1/2}\cdot \frac{t^2}m\sum_{k=0}^{m-1}|\<\psi,\Psi_{u,t}(x_k)\>_{\bH_\calY}|\|[\wh C_\calX^m]_r^\dagger\Phi(x_k)\|_{\bH_\calX}\\
&\le c_r\veps\|\psi\|_{\bH_\calY} + \|k\|_{\infty}^{1/2}\cdot \frac{t^2}m\sum_{k=0}^{m-1}\|\psi\|_{\bH_\calY}\|\Psi_{u,t}(x_k)\|_{\bH_\calY}\wh\la_r^{-1}\sqrt{\vphi(x_k)}\\
&\le c_r\veps\|\psi\|_{\bH_\calY} + \wh\la_r^{-1}C_\Phi\|k\|_\infty\cdot t^2\|\psi\|_{\bH_\calY},
\end{align*}
which proves the claim with $C_k = C_\Phi\|k\|_\infty$.
\end{proof}

\section{Numerical results}\label{s:num}
In this part, we present two numerical examples. First we consider an autonomous stochastic system to illustrate the results of Section~\ref{sec:res1}. Second, we present numerical results for a controlled deterministic system to showcase the bounds deduced in Section~\ref{s:control}.
\subsection{Estimation Error for Ornstein-Uhlenbeck Process}
We consider a one-dimensional Ornstein-Uhlenbeck (OU) process obeying the stochastic differential equation
\begin{equation}
\label{eq:ou_process}
    d X_t  = -\alpha X_t  + \sqrt{2\beta^{-1}} dW_t,
\end{equation}
with $\alpha = 1$ and $\beta = 2$ in the numerical tests. The initial domain is chosen as $\calX = [-1.5, 1.5]$, the final domain is simply the entire real line $\calY = \R$. The measure $\mu$ is chosen as the invariant distribution of the OU process on $\calY$, which is normal with mean zero and standard deviation $(\alpha\beta)^{-1/2}$. We then use the restriction of $\mu$ to the finite interval $[-1.5, 1.5]$ to generate initial conditions for the SDE~\eqref{eq:ou_process}.

Mirroring the analysis in~\cite{PhilScha23}, we consider the approximation of the cross-covariance operator $C^t_{\calX\calY}$ at lag time $t = 0.05$, corresponding to five elementary time steps at integration time step $\Delta_t = 0.01$. We use the Gaussian radial basis function kernel
\begin{equation*}
k(x, y) = \exp\Big[-\frac{(x - y)^2}{\sigma^2}\Big],
\end{equation*}
for bandwidths $\sigma \in \{0.05, 0.1, 0.5\}$. Setting the confidence level to $1 - \delta = 0.9$, we use Proposition~\ref{p:prob_est} to bound the minimal guaranteed error that is achieved with this confidence, for a range of data sizes $m$ between $m = 20$ and $m = 5 \cdot 10^4$. As a comparison, we estimate the actual incurred error for fifty independent collections of $m$ pairs $(x_k, y_k)$ each, and plot the $(1 - \delta$)-percentile for this error as a function of $m$. To do so, we first observe that the Hilbert-Schmidt norm of the empirical covariance operator $\wh{C}^{t, m}_{\calX\calY}$ can be calculated using the Mercer features and eigenvalues of the Gaussian kernel:
\begin{equation}
\label{e:hs_norm_ct}
\|\wh{C}^{t,m}_{\calX\calY} \|^2_{\mathrm{HS}} = \sum_{i, j=1}^\infty \lambda_i \lambda_j (M^{m}_{ij})^2,
\end{equation}
using the following matrix elements for the Mercer features:
\begin{equation}
\label{e:matrix_element_mercer}
    M^{m}_{ij} := \frac{1}{m}\sum_{l=1}^{m} e_i(x_l)e_j(y_l).
\end{equation}
In practice, we truncate the infinite sum in~\eqref{e:hs_norm_ct} after $i_{\max} = j_{\max} = 30$ terms. We then first generate ten large reference collections $(x_l, y_l), 1\leq l \leq m_{\max}$, where $m_{\max} = 10^7$, and average the matrix elements $M^{m_{\max}}_{ij}$ over these large reference samples. The actual error for $m < m_{\max}$ data pairs $(x_k, y_k), 1\leq k \leq m$ is then estimated by the formula
\begin{equation*}
\|C^t_{\calX\calY} - \wh{C}^{t,m}_{\calX\calY}\|^2_{\mathrm{HS}} \approx \|\wh{C}^{t,m_{\max}}_{\calX\calY} - \wh{C}^{t,m}_{\calX\calY}\|^2_{\mathrm{HS}} = \|\wh{C}^{t,m_{\max}}_{\calX\calY}\|^2_{\mathrm{HS}} - 2 \innerprod{\wh{C}^{t,m_{\max}}_{\calX\calY}}{\wh{C}^{t, m}_{\calX\calY}}_{\mathrm{HS}} + \|\wh{C}^{t,m}_{\calX\calY} \|^2_{\mathrm{HS}}.
\end{equation*}
The Hilbert-Schmidt inner product can be computed by the formula
\begin{equation*}
    \innerprod{\wh{C}^{t,m_{\max}}_{\calX\calY}}{\wh{C}^{t, m}_{\calX\calY}}_{\mathrm{HS}} = \sum_{i, j=1}^\infty \lambda_i \lambda_j M^{\max}_{ij} M^m_{ij}.
\end{equation*}
The results are shown in Figure~\ref{fig:estimation_error_ou}. On average, the bounds from Proposition~\ref{p:prob_est} over-estimate the actual incurred error by about one order of magnitude.

\begin{figure}[htb]
    \centering
    \includegraphics[width=0.65\textwidth]{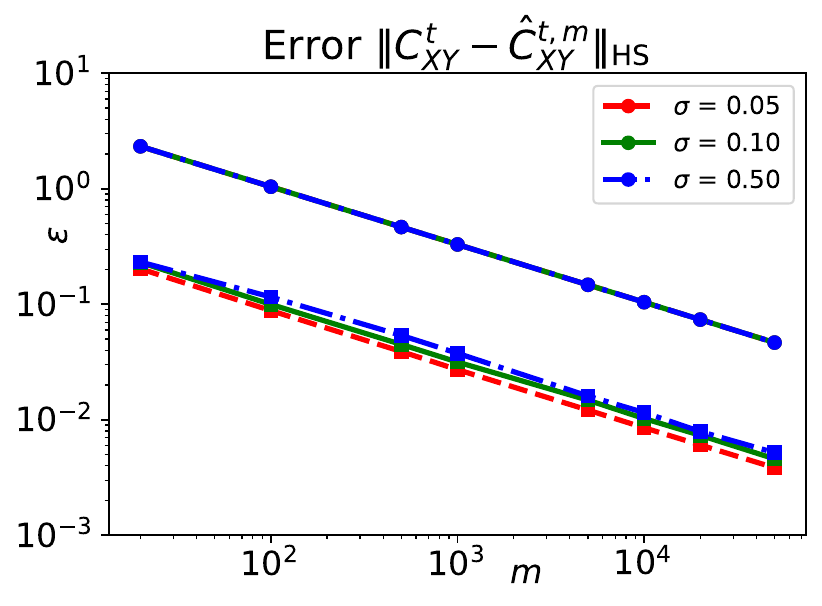}
    \caption{Ornstein-Uhlenbeck process: Error bound from Prop.~\ref{p:prob_est} (circles) versus $(1-\delta)$-percentile of the actual estimation error estimated from fifty independent trials (squares), for three different bandwidths $\sigma$ indicated by different colors and linestyles.}
    \label{fig:estimation_error_ou}
\end{figure}

\subsection{Duffing Oscillator}
To conclude this study, we illustrate that kernel EDMD, combined with state-of-the art model validation and low-rank approximation techniques, enables accurate prediction of a non-linear control system with very few necessary design choices on the modeler's end. We study the two-dimensional Duffing oscillator, governed by the ordinary differential equation
\begin{equation}
\label{eq:duffing}
    \frac{\mathrm{d}}{\mathrm{d} t}\begin{bmatrix}
        z_1(t)\\z_2(t)
    \end{bmatrix} = \begin{bmatrix}
        z_2 \\ -\alpha z_1 u - 2\beta z_1^3
    \end{bmatrix},
\end{equation}
which has been used as a model system in numerous previous studies. We set $\alpha = -1$, $\beta = 1$. By $z(t;z^0,u)$, we denote the state at time $t\geq 0$ emanating from the initial condition $z^0$ at time zero and the control $u:\R^{\geq 0}\to \R$. We generate $m = 10^4$ uniformly sampled initial conditions $\{x_k\}_{k=1}^m$ from the square $\calX = [-1.5, 1.5]^2$. Next, for fixed control inputs $\bar{u} \in \{0, 1\}$, we compute the successor state $y_k^{\bar{u}} = z(t;x_k,\bar{u}), k=1,\ldots,m$ for $t=0.025$ via numerical integration with time step $\Delta_t = 0.005$. The data pairs $\{(x_k, y_k^{\bar u})\}_{k=1}^{m}$ serve as training data for the kEDMD algorithm corresponding to the flow with constant control $\bar{u}\in \{0,1\}$. As shown in Figure~\ref{fig:duff_training}, the time-shifted points $y_k^{\bar{u}}$ can indeed exit the initial domain $\calX$, such that invariance of $\calX$ is not given. This in particular underlines the necessity of incorporating an additional set $\calY$ as done in the previous analysis.

\begin{figure}[htb]
    \centering
    \includegraphics[width=0.48\textwidth]{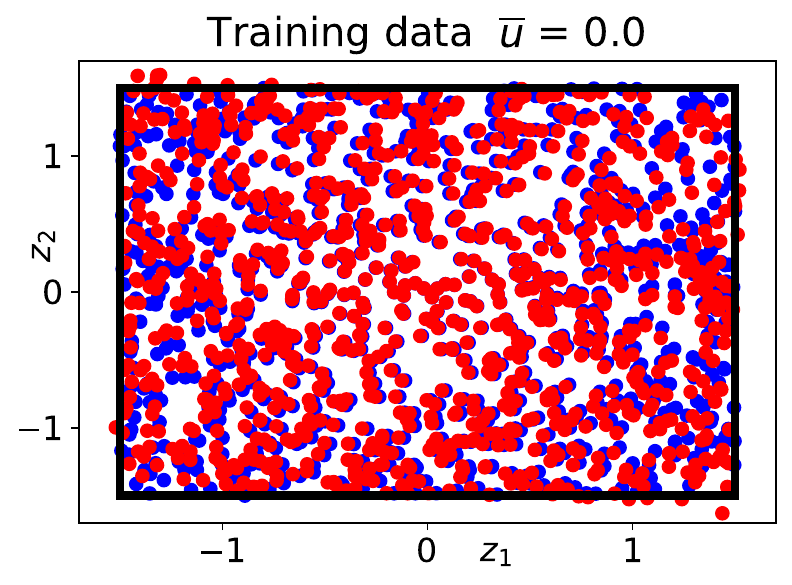}
    \includegraphics[width=0.48\textwidth]{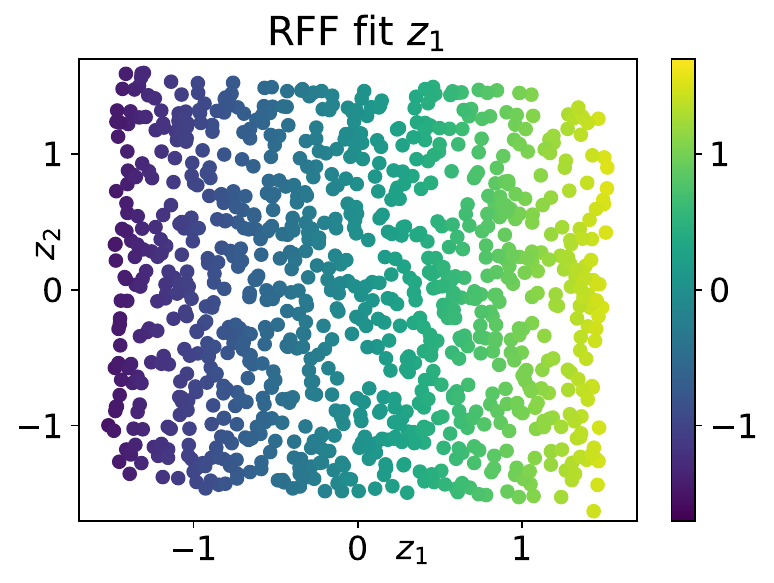}
    \caption{Left: initial training data $x_k$ (blue) and corresponding time shifted data $y_k$ (red) for fixed input $\bar u = 0$. The black square indicates the initial domain $\calX = [-1.5, 1.5]^2$. Right: Evaluation of the state variable $x_k$ on the data $y_k^{\bar u}$, approximated in terms of the RFF basis at $\sigma = 1.0$ and $p = 500$.}
    \label{fig:duff_training}
\end{figure}

 As reproducing kernel, we again select the Gaussian RBF kernel. To reduce the computational effort, we approximate the resulting feature maps using random Fourier features (RFFs)~\cite{rahimi2007}. It was shown in~\cite{nueske23_rff} that this is equivalent to applying plain EDMD with a randomly selected basis set of complex exponentials
\begin{equation*}
    \psi_i(x) = e^{i\omega_i^\top x},\quad \ 1\leq i \leq p\in \N,
\end{equation*}
where $\omega_i \in \R^d, \, 1\leq i \leq p$, are random frequencies drawn from the spectral measure associated with the kernel. In our case, the spectral measure is a normal distribution with standard deviation $\sigma^{-1}$, i.e. the inverse of the kernel bandwidth. Using the feature matrices {$\mathbf{M},\mathbf{M}_t^{\bar u}\in \C^{p \times m}$ defined by
\begin{align*}
    \mathbf{M}_{i,k} = e^{i\omega_i^\top x_k} ,\quad  (\mathbf{M}_t^{\bar u})_{i,k} = e^{i\omega_i^\top y_k^{\bar u}}, \quad 1\leq i\leq p,\ 1\leq k \leq m,
\end{align*}
the RFF Koopman matrix corresponding to the flow with constant control $\bar{u}\in \{0,1\}$ amounts to
\begin{equation*}
    \mathbf{K}_{\bar{u}}
    = \left(\mathbf{M}\mathbf{M}^{\rm H} + \gamma {\rm Id} \right)^{-1} \mathbf{M}(\mathbf{M}_t^{\bar u})^{\rm H},
\end{equation*}
where $\gamma > 0$ is a regularization parameter.

We begin by tuning the bandwidth $\sigma>0$, the feature size $p\in \N$, and the tolerance $\gamma > 0$. To this end, we draw a test ensemble of $100$ random initial conditions, from which we generate simulations of the ODE~\eqref{eq:duffing} until time $T_{\rm test} = 20t = 0.5$. We also use the learned Koopman models to predict the system state over the same time horizon, starting from the same initial conditions, and compute the relative mean-squared error between the true and predicted states over all time steps. Note that state prediction with the Koopman operator requires an approximation of the state variables $z_1$ and $z_2$ as linear combinations of the RFF basis functions $\psi_i$. These approximations are computed by a least-squares fit along with each Koopman model. An example is shown in Figure~\ref{fig:duff_training}.

The results of the model validation are shown for different inputs $u:[0,T_{\mathrm{test}}]\to \R$ and different hyper-parameters in Figure~\ref{fig:duffing_validation}. We conclude that the prediction error generally decreases with regularization if $\gamma \geq 10^{-5}$, but seems to destabilize if the regularization is decreased further. We therefore identify $\sigma =1.0$, $p = 500$ and $\gamma = 10^{-5}$ as suitable model parameters. 

\begin{figure}
    \centering
    \includegraphics[width=0.48\textwidth]{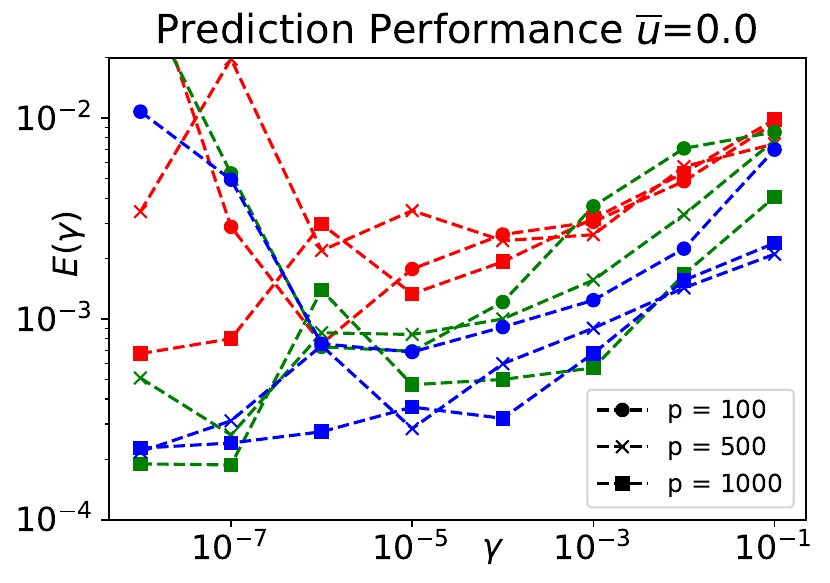}
    \includegraphics[width=0.48\textwidth]{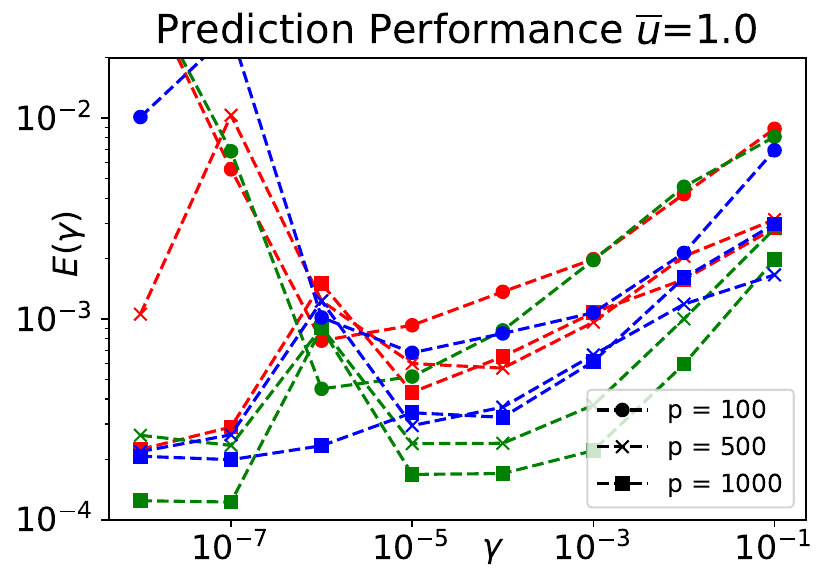}
    \caption{Duffing oscillator: Mean-squared prediction error of the trained Koopman model for $\bar{u} = 0.0$ (left) and $\bar{u} = 1.0$ (right), tested on $100$ new trajectories of length $T_{\rm test} = 20t = 0.5$. The error is shown as a function of the regularization strength $\gamma>0$ for different feature sizes $p\in \N$ (dots, x-es, squares) and bandwiths $\sigma \in \{0.5, 1.0, 1.5\}$ (blue, green, red).}
    \label{fig:duffing_validation}
\end{figure}

We then evaluate the long-time prediction performance of the bi-linear Koopman surrogate model
\begin{equation}
\label{eq:koopman_bilin}
    \mathbf{K}_{u(t)}\mathbf{w} = \left[\mathbf{K}_0 + (\mathbf{K}_1 - \mathbf{K}_0)u(t)\right]\mathbf{w}
\end{equation}
for the lifted state $\mathbf{w}\in \C^p$, generating $50$ trajectories for (a) constant control input $u \equiv 0$; (b) constant control input $u \equiv 1$; (c) time-dependent control input $u(t) = \cos(t)$. To obtain predictions corresponding to the real-valued system \eqref{eq:duffing}, we project the propagated complex-valued lifted state onto its real part. The time horizon covered by these trajectories is $T_{\rm long} = 500t = 12.5$, and we use the project-and-lift scheme~\cite{GoorMaho23} every 25 steps to improve performance. Figure~\ref{fig:duffing_long_time} shows the mean relative error over the real part of these trajectories as a function of simulation time. We also show exemplary trajectories for each of the three scenarios in the same figure, and observe excellent performance of the kernel-based bi-linear Koopman model.

\begin{figure}[htb]
    \centering
    \includegraphics[width=0.48\textwidth]{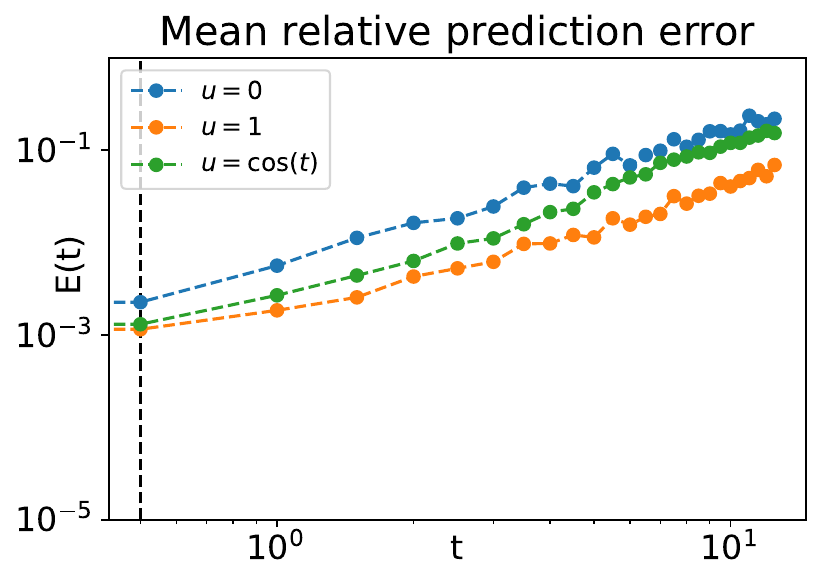}
    \includegraphics[width=0.48\textwidth]{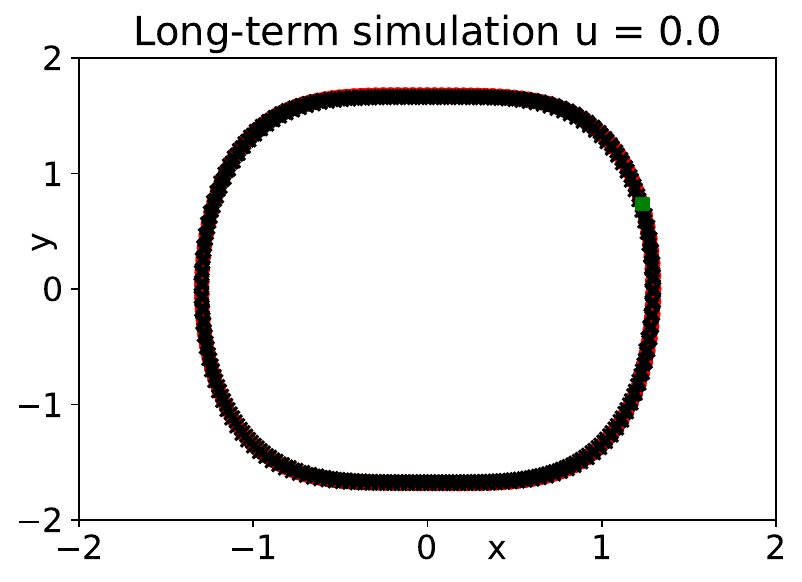}
    \includegraphics[width=0.48\textwidth]{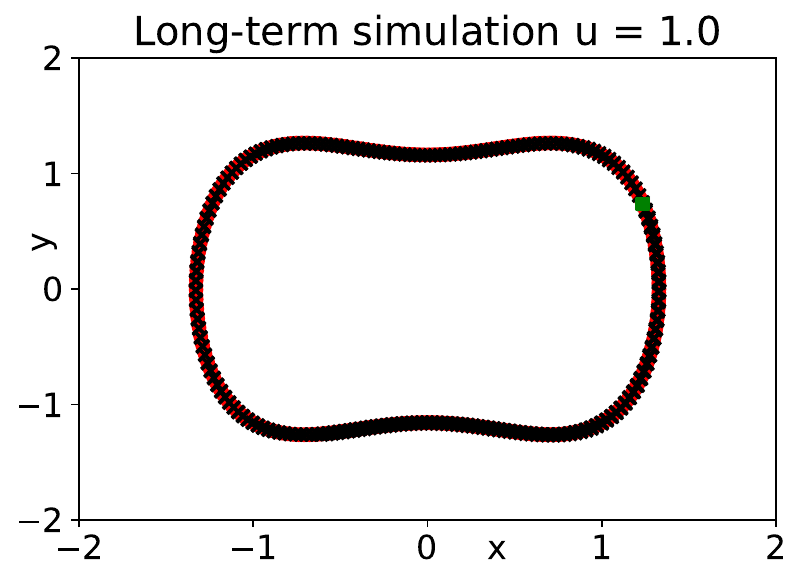}
    \includegraphics[width=0.48\textwidth]{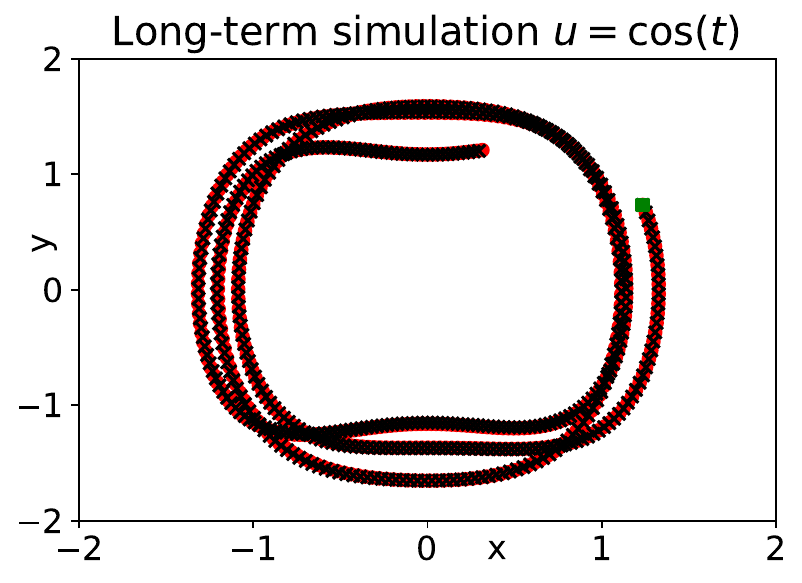}
    \caption{Upper left: Mean-squared prediction error for the Duffing oscillator predicted by the bi-linear Koopman model~\eqref{eq:koopman_bilin} for different control inputs, evaluated on $50$ new trajectories of length $T_{\rm long} = 500t = 12.5$. The black dashed line indicates the time horizon $T_{\rm test}$ used for model validation in Figure~\ref{fig:duffing_validation}. The remaining panels show exemplary trajectories of the same length $T_{\rm long}$ for different control inputs. Red dots mark the predicted solution, the true ODE solution is shown in black. The initial value is highlighted in green. }
    \label{fig:duffing_long_time}
\end{figure}

\section{Supplements}
\noindent{\bf Acknowledgments.} F. Philipp was funded by the Carl Zeiss Foundation within the project DeepTurb--Deep Learning in and from Turbulence. He was further supported by the free state of Thuringia and the German Federal Ministry of Education and Research (BMBF) within the project THInKI--Th\"uringer Hochschulinitiative für KI im Studium. 
K. Worthmann gratefully acknowledges funding by the German Research Foundation (DFG grant WO 2056/14-1, project no 507037103).

\medskip
\noindent{\bf Data Availability.} Codes and data required for reproducing the numerical results in this study are available at \url{https://doi.org/10.5281/zenodo.10378153}.

\bigskip
\appendix

\section{The condition \texorpdfstring{\eqref{e:LipMeas}}{x} for deterministic systems}\label{a:deterministic}
Let $\calX = \ol\Omega$, where $\Omega\subset\R^d$ is open and bounded. We consider the flow $F : [0,T]\times\calX\to\R^d$ of a given autonomous ODE $\dot x = b(x)$. Let $\calY\subset\R^d$ such that
\[
\calZ := F([0,T]\times\calX)
\]
is contained in $\calY$. In this section, we provide conditions on the flow and a finite measure $\mu$ on $\calY$ which guarantee that condition \eqref{e:LipMeas_det} holds, i.e.,
\[
\mu(F_t^{-1}(A))\,\le\,L\cdot\mu(A)\quad\forall A\in\calB(\calY),\,\forall t\in [0,T].
\]
First of all, we note that it is sufficient to prove the above for $\calY$ replaced by $\calZ$. Indeed, then for $A\in\calB(\calY)$ we have
\[
\mu(F_t^{-1}(A)) = \mu(F_t^{-1}(A\cap\calZ))\le L\cdot\mu(A\cap\calZ)\le L\cdot\mu(A).
\]

\begin{prop}\label{p:deterministic}
The condition \eqref{e:LipMeas_det} is satisfied whenever the following conditions are met:
\begin{itemize}
\item The flow $F : [0,\wt T]\times\calX\to\R^d$ of the ODE $\dot x = b(x)$ is defined for some $\wt T > T$ and is $C^1$.
\item $d\mu(x) = \vphi(x)\,dx$ with a positive density $\vphi\in L^\infty(\calX)$ satisfying $1/\vphi\in L^\infty(\calZ)$.
\end{itemize}
\end{prop}

Note that the first assumption is satisfied if $b\in C^1(\calY)$, cf.\ \cite[Theorem 10.3]{amann}.

\begin{proof}[Proof of Proposition \ref{p:deterministic}]
First of all, we show that
\begin{align}\label{e:delta}
\delta := \inf\big\{\det F_t'(x) : x\in\calX,\,t\in [0,T]\big\} > 0.
\end{align}
Since $F_0(z) = z$ for all $z\in\calZ$, we have $F_0'(z) = I_d$ for all $z\in\calZ$. Let $t_1 := \wt T - T$. By assumption, the flow $F : [0,t_1]\times\calZ\to\R^d$ is defined on $\calZ$ and is $C^1$. Hence, there exists some $t_0\in (0,t_1)$ such that $\det F_t'(z)\ge\frac 12$ for all $t\in [0,t_0]$ and all $z\in\calZ$. Choose $t_0$ such that $T = nt_0$ for some $n\in\N$. We prove by induction that for all $k\in\{0,\ldots,n-1\}$ we have
\begin{align}\label{e:induction}
\forall\,t\in [0,t_0]\,\forall x\in\calX : \det F_{kt_0+t}'(x)\ge 2^{-k-1}.
\end{align}
Then \eqref{e:delta} holds with $\delta = 2^{-n}$. We have proved \eqref{e:induction} already for $k=0$. So, let \eqref{e:induction} hold for some $k < n-1$, and let $t\in [0,t_0]$ and $x\in\calX$. Then
\[
\det F_{(k+1)t_0+t}'(x) = \det D_x F_{t_0}(F_{kt_0+t}(x)) = \det F_{t_0}'(F_{kt_0+t}(x))\cdot\det F_{kt_0+t}'(x)\,\ge\,\tfrac 12\cdot\tfrac 1{2^{k+1}} = \tfrac 1{2^{k+2}},
\]
which concludes the proof of \eqref{e:delta}.

In order to prove \eqref{e:LipMeas_det}, we may restrict the flow to $F : [0,T]\times\Omega\to\R^d$ as $\mu(\partial\Omega)=0$. Let $t\in [0,T]$, set $\calZ_t := F_t(\Omega)$, and let $U\subset\R^d$ be open. Note that $\calZ_t$ is open due to the invariance of domain theorem\footnote{or simply the inverse function theorem thanks to \eqref{e:delta}}. Hence, also $U_t = U\cap\calZ_t$ is open. By the change of variables formula we have
\begin{align*}
\mu(F_t^{-1}(U))
&= \mu(F_t^{-1}(U_t)) = \int_{F_t^{-1}(U_t)}\vphi(x)\,dx = \int_{U_t}\vphi(F_t^{-1}(y))|\det DF_t^{-1}(y)|\,dy\\
&= \int_{U_t}\frac{\vphi(F_t^{-1}(y))}{|\det F_t'(F_t^{-1}(y))|\vphi(y)}\vphi(y)\,dy\,\le\,\frac{\|\vphi\|_\infty\|\frac 1{\vphi}\|_\infty}{\delta}\mu(U_t) \le \frac{\|\vphi\|_\infty\|\frac 1{\vphi}\|_\infty}{\delta}\mu(U).
\end{align*}
The claim now follows from the outer regularity of the measures $\mu$ and $\mu\circ F_t^{-1}$.
\end{proof}

\section{Proof of Proposition \ref{p:Kbounded}}\label{a:proof_Kt_bounded}
Let $\psi\in B(\calY)$. Then $|(K^t\psi)(x)|\le\int_\calY|\psi(y)|\,\rho_t(x,dy)\le\|\psi\|_\infty$, hence $K^t\psi\in B(\calX)$ with $\|K^t\psi\|_\infty\le\|\psi\|_\infty$ and
$\|K^t\psi\|_p\le\|\psi\|_\infty$. Let now $1\le p<\infty$. Then, by H\"older's inequality and \eqref{e:LipMeas_Func},
$$
\int_\calX|(K^t\psi)(x)|^p\,d\mu(x)\le\int_\calX\int_\calY|\psi(y)|^p\,\rho_t(x,dy)\,d\mu(x)\le L\int_\calY|\psi(y)|^p\,d\mu(y),
$$
and hence $\|K^t\psi\|_{p,\calX}\le L^{1/p}\|\psi\|_{p,\calY}$.

Let $\psi\in C_c(\calY)$ (the set of continuous functions on $\calY$ with compact support). 
For fixed $x\in \calX$, denote the stochastic solution process of the SDE \eqref{e:SDE0} with initial value $x$ by $X_t^x$. Since $X_t^x(\omega)$ is continuous in $t$ a.s.\ (see \cite[Theorem 5.2.1]{oe}) and $X_t^x\in\calY$ a.s., we have $\psi(X_t^x(\omega))\to\psi(X_0^x(\omega)) = \psi(x)$ as $t\to 0$ a.s.\ Hence, by dominated convergence,
$$
K^t\psi(x) = \bE[\psi(X_t^x)] = \int\psi(X_t^x(\omega))\,d\bP(\omega)\,\to\,\psi(x)
$$
as $t\to 0$. By the first part of this proof,
$$
|K^t\psi(x) - \psi(x)|\,\le\,|K^t\psi(x)| + |\psi(x)|\le 2\|\psi\|_\infty,\qquad x\in\calX,
$$
hence, dominated convergence implies that $\|K^t\psi - \psi|_\calX\|_{p,\calX}\to 0$ as $t\to 0$.

If $\psi\in L^p(\calY;\mu)$ and $\veps>0$, there exists $\eta\in C_c(\calY)$ such that $\|\psi-\eta\|_{p,\calY} < \frac \veps 2(1+L^{1/p})^{-1}$, see \cite[Theorem 3.14]{rrca}. Further, choose $\delta>0$ such that $\|K^t\eta-\eta|_\calX\|_{p,\calX} < \frac{\veps}2$ for $t < \delta$. Then
$$
\|K^t\psi - \psi|_\calX\|_{p,\calX}\le \|K^t(\psi - \eta)\|_{p,\calX} + \|K^t\eta-\eta|_\calX\|_{p,\calX} + \|\eta-\psi\|_{p,\calY} < \veps
$$
for $t < \delta$, which proves the claim.\qed

\section{An Example: Lack of contractivity of the Koopman semigroup}\label{s:unbounded}
\noindent Consider the deterministic case with the scalar ODE
\begin{align*}
    \dot{x} = x(2-x)\quad\text{on $\calX = [1,2]$}.
\end{align*}
The solution of this ODE is given by
\begin{align*}
x(t;x_0) = \frac{2e^{2t}x_0}{2-x_0+e^{2t}x_0},\quad x_0\in\calX,
\end{align*}
which is non-decreasing and approaches $2$ as $t\to\infty$ for any $x_0\in\calX$. In particular, $\calX$ is invariant under the flow of the ODE. Since $\calX$ is compact, the convergence towards the equilibrium $x=2$ is uniform in $x_0\in\calX$. We will use this fact to show that the Koopman semigroup is not a bounded semigroup on $L^2(\calX)$, i.e., it does not satisfy $\|K^t\|_{L(L^2(\calX))}\leq M$ for some $M\geq 0$ and all $t\ge 0$. To show this, it suffices to find $\psi\in L^2(1,2)$ such that for any $M\geq 0$, there is $t\geq 0$ such that $\|K^t\psi\|_{L^2(\calX)}\geq M\|\psi\|_{L^2(\calX)}$.

Let $\psi(x) := \frac{1}{(2-x)^{1/4}}$. Then $\psi\in L^2(\calX)$ as 
\begin{align*}
\int_1^2\psi^2(x)\,\mathrm{d}x = \int_1^2 \frac{1}{\sqrt{2-x}}\,\mathrm{d}x = \int_0^1 \frac{1}{\sqrt{x}}\,\mathrm{d}x = 2\sqrt x\Big|_0^1 = 2 < \infty.
\end{align*}
Further, we have that $\psi(x)\to\infty$ as $x\to 2$ and $\psi$ is non-decreasing such that 
\begin{align}\label{ctr:Meps}
\forall M\geq 0 \,\,\exists\,\veps>0:\,\psi(x) \geq M \qquad \forall x\in (2-\varepsilon,2).
\end{align}
Moreover, as the flow is non-decreasing towards $2$:
\begin{align}\label{ctr:epst}
\forall\veps\geq 0\,\,\exists\,t^*\geq 0\,\,\forall t \geq t^*\,\,\forall x_0\in\calX\backslash\{2\}:\, x(t;x_0)\in (2-\varepsilon,2).
\end{align}
Hence, for given $M$, choose $\varepsilon$ as in \eqref{ctr:Meps} and correspondingly $t^*$ as in \eqref{ctr:epst}. Then $x(t;x^0)\in (2-\varepsilon,2)$ for all $t\ge t^*$ and $x_0\in\calX\backslash\{2\}$, hence
\begin{align*}
\forall t\geq t^*:\quad  \|K^t\psi\|^2_{L^2(\calX)} = \int_1^2 |\psi(x(t;x_0))|^2\,\mathrm{d}x_0 \ge M^2 = \tfrac 12 M^2\|\psi\|_{L^2(\calX)}^2.
\end{align*}

\section{On the condition \texorpdfstring{$K^t\bH_\calY\subset\bH_\calX$}{x} in case of the Gaussian kernel and the Ornstein-Uhlenbeck process}\label{a:OU_invariance}
The Ornstein-Uhlenbeck (OU) process on $X = \R$ is the solution of the SDE $dX_t = -\alpha X_t\,dt + dW_t$, where $\alpha>0$. The invariant measure $\mu$ and the Markov transition kernel $\rho_t$, $t>0$, are known and given by
\begin{align*}
d\mu(x) = \sqrt{\frac{\alpha}{\pi}}\cdot e^{-\alpha x^2}\,dx
\qquad\text{and}\qquad
\rho_t(x,dy) = \sqrt{\frac{c_t}{\pi}}\cdot\exp\Big[-c_t(y - e^{-\alpha t}x)^2\Big]\,dy,
\end{align*}
where
$$
c_t = \frac{\alpha}{1 - e^{-2\alpha t}}.
$$
We consider the Gaussian radial basis function (RBF) kernels on $\R$ with bandwidth $\sigma>0$, i.e.,
\begin{equation*}
    k^\sigma(x, y) = \exp\left[-\frac{(x-y)^2}{\sigma^2}\right].
\end{equation*}
For $y\in\R$ and a kernel $k$ on $\R$ set $k_y(x) := k(x,y)$, $x\in\R$.

Let $X\subset\R$ be any set. By $\bH_{\sigma}(X)$ we denote the RKHS generated by the kernel $k^\sigma_X = k^\sigma|_{X\times X}$ on $X$. Hilbert space norm and scalar product on $\bH_{\sigma}(X)$ will be denoted by $\|\cdot\|_{\sigma,X}$ and $\<\cdot\,,\,\cdot\>_{\sigma,X}$, respectively. For two positive definite kernels on $X$ we write $k_1\preceq k_2$ if
$$
\sum_{i,j=1}^n\alpha_i\alpha_j k_1(x_i,x_j)\le \sum_{i,j=1}^n\alpha_i\alpha_j k_2(x_i,x_j)
$$
for any choice of $n\in\N$ and $\alpha_j\in\R$, $x_j\in X$, $j=1,\ldots,n$. We also write $V\lhook\joinrel\xrightarrow{\hspace{.1cm}c\hspace{.1cm}}W$ for two normed vector spaces $V$ and $W$ if $V\subset W$ is continuously embedded in $W$.

\begin{lem}\label{l:rbf}
Let $X\subset\R$ and $0 < \sigma_1 < \sigma_2$. Then we have  $\bH_{\sigma_2}(X)\!\lhook\joinrel\xrightarrow{\hspace{.1cm}c\hspace{.1cm}}\bH_{\sigma_1}(X)$ with
\[
\|\psi\|_{\sigma_1,X}\le\sqrt{\tfrac{\sigma_2}{\sigma_1}}\cdot\|\psi\|_{\sigma_2,X},\qquad\psi\in\bH_{\sigma_2}(X),
\]
and $k_X^{\sigma_2}\preceq\frac{\sigma_2}{\sigma_1}k_X^{\sigma_1}$.
\end{lem}
\begin{proof}
The first claim is Corollary 6 in \cite{shs}. The second follows from Aronszajn's inclusion theorem \cite[Theorem 5.1]{pr}.
\end{proof}

\begin{thm}\label{t:OU_invariant}
Let $\calY = \R$ and let $\calX\subset\R$ be any compact set with non-empty interior. Then for each $t\ge 0$ and all $\alpha,\sigma > 0$ we have 
\[
K^t\bH_{\sigma}(\calY)\subset\bH_{\sigma}(\calX)
\qquad\text{with}\qquad
\big\|K^t\big\|_{\bH_{\sigma}(\calY)\to\bH_{\sigma}(\calX)}\le e^{\frac\alpha 2t}.
\]
\end{thm}
\begin{proof}
First of all, it follows from \cite[Corollaries 4\&5]{shs} that to each $\psi\in\bH_\sigma(\calX)$ there exists a unique extension $\tilde\psi\in\bH_\sigma(\calY)$ such that $\<\tilde\psi,\tilde\phi\>_{\sigma,\calY} = \<\psi,\phi\>_{\sigma,\calX}$ for all $\psi,\phi\in\bH_\sigma(\calX)$. Conversely, we have
\begin{align}\label{e:sp_inclusion}
\bH_\sigma(\calX) = \{\psi|_\calX : \psi\in\bH_\sigma(\calY)\}
\qquad\text{and}\qquad
\<\psi,\phi\>_{\sigma,\calY} = \<\psi|_\calX,\phi|_\calX\>_{\sigma,\calX},\qquad \psi,\phi\in\bH_\sigma(\calY).
\end{align}
For $t=0$ the claim is obviously true. Hence, suppose that $t>0$. For $z\in\R$ let us compute $K^tk^\sigma_z$. For this, we define
$$
\tau = \sqrt{\frac{c_t\sigma^2}{1+c_t\sigma^2}},
\qquad\text{and}\qquad
\nu = \frac{e^{\alpha t}}{\tau}\sigma\,>\,\frac\sigma\tau > \sigma.
$$
Since for $\sigma_1,\sigma_2>0$ and $z,w\in\R$ we have
\[
\int_{-\infty}^\infty k_z^{\sigma_1}(x)\,k_w^{\sigma_2}(x)\,dx = \sqrt{\pi}\cdot\frac{\sigma_1\sigma_2}{\sqrt{\sigma_1^2 + \sigma_2^2}}\cdot\exp\left(-\frac{(z-w)^2}{\sigma_1^2+\sigma_2^2}\right),
\]
with $\sigma_t = 1/\sqrt{c_t}$ we obtain for $x\in\calX$
\begin{align*}
(K^tk_z^\sigma)(x)
&= \int_{-\infty}^\infty k_z^\sigma(y)\,\rho_t(x,dy) = \sqrt{\frac{c_t}{\pi}}\int_{-\infty}^\infty k_z^\sigma(y)\exp\Big[-c_t(y - e^{-\alpha t}x)^2\Big]\,dy\\
&= \sqrt{\frac{c_t}{\pi}}\int_{-\infty}^\infty k_z^\sigma(y)k_{e^{-\alpha t}x}^{\sigma_t}(y)\,dy = \sqrt{c_t}\cdot\frac{\sigma\sigma_t}{\sqrt{\sigma^2 + \sigma_t^2}}\cdot \exp\left(-\frac{(z-e^{-\alpha t}x)^2}{\sigma^2+\sigma_t^2}\right)\\
&= \frac{\sigma}{\sqrt{\sigma^2 + 1/c_t}}\cdot\exp\left(-\frac{(e^{\alpha t}z-x)^2}{e^{2\alpha t}(\sigma^2+1/c_t)}\right) = \tau\cdot k^\nu_{e^{\alpha t}z}(x).
\end{align*}
Since $\nu>\sigma$, it follows from \eqref{e:sp_inclusion} and Lemma \ref{l:rbf} that
$$
K^tk_z^\sigma = \tau\cdot k^\nu_{e^{\alpha t}z}|_\calX\in\bH_{\nu}(\calX)\subset\bH_\sigma(\calX).
$$
Note that
$$
k^\nu(e^{\alpha t}x,e^{\alpha t}y) = \exp\left(-\frac{e^{2\alpha t}(x-y)^2}{\nu^2}\right) = \exp\left(-\frac{\tau^2(x-y)^2}{\sigma^2}\right) = k^{\sigma/\tau}(x,y).
$$
Now, let $n\in\N$ and $\alpha_j,x_j\in\R$, $j=1,\ldots,n$, be arbitrary and let $\psi = \sum_{j=1}^n\alpha_jk^\sigma_{x_j}\in\bH_\sigma(\calY)$. Then, again by \eqref{e:sp_inclusion} and Lemma \ref{l:rbf},
\begin{align*}
\|K^t\psi\|_{\sigma,\calX}^2
&\le\frac{\nu}{\sigma}\|K^t\psi\|_{\nu,\calX}^2
= \frac{\nu}{\sigma}\Bigg\|\sum_{j=1}^n\alpha_jK^tk^\sigma_{x_j}\Bigg\|_{\nu,\calX}^2 = \frac{\nu\tau^2}{\sigma}\Bigg\|\sum_{j=1}^n\alpha_jk^\nu_{e^{\alpha t}x_j}\Bigg\|_{\nu,\calY}^2\\
&= \frac{\nu\tau^2}{\sigma}\sum_{i,j=1}^n\alpha_i\alpha_jk^\nu(e^{\alpha t}x_i,e^{\alpha t}x_j) = \frac{\nu\tau^2}{\sigma}\sum_{i,j=1}^n\alpha_i\alpha_jk^{\sigma/\tau}(x_i,x_j)\\
&\le \frac{\nu\tau^2}{\sigma}\cdot\frac 1\tau\sum_{i,j=1}^n\alpha_i\alpha_jk^{\sigma}(x_i,x_j) = \frac{\nu\tau}\sigma\|\psi\|_{\sigma,\calY}^2 = e^{\alpha t}\|\psi\|_{\sigma,\calY}^2.
\end{align*}
This shows that $K^t$ maps $\bH_{0,\sigma}(\calY) := \linspan\{k^\sigma_y : y\in\calY\}\subset\bH_\sigma(\calY)$ boundedly into $\bH_\sigma(\calX)$. Since $\bH_{0,\sigma}(\calY)$ is dense in $\bH_\sigma(\calY)$, it follows that $K^t|_{\bH_{0,\sigma}}$ extends to a bounded operator $T : \bH_\sigma(\calY)\to\bH_\sigma(\calX)$. In order to see that $T\psi = K^t\psi$ for $\psi\in\bH_\sigma(\calY)$, let $(\psi_n)\subset\bH_{0,\sigma}(\calY)$ such that $\psi_n\to\psi$ in $\bH_\sigma(\calY)$. Then $K^t\psi_n = T\psi_n\to T\psi$ in $\bH_\sigma(\calX)$. Since $\bH_\sigma(\calY)\lhook\joinrel\xrightarrow{\hspace{.1cm}c\hspace{.1cm}}L^2_\mu(\calY)$, we have $\psi_n\to\psi$ in $L^2_\mu(\calY)$ and thus $K^t\psi_n\to K^t\psi$ in $L^2_\mu(\calX)$. Also, $K^t\psi_n\to T\psi$ in $L^2_\mu(\calX)$. Hence, $K^t\psi=T\psi$ $\mu$-a.e.\ on $\R$. But as both $K^t\psi$ and $T\psi$ are continuous and $\mu$ is absolutely continuous w.r.t.\ Lebesgue measure with a positive density, we conclude that $K^t\psi = T\psi\in\bH_\sigma$.
\end{proof}

\bigskip
\section*{Author Affiliations}
\end{document}